\documentclass[reqno]{amsart}

\usepackage{amsmath,tabu}
\usepackage{amsfonts}
\usepackage{amssymb}
\usepackage{amsthm}
\usepackage[bookmarks]{hyperref}
\hypersetup{colorlinks=false}
\usepackage{cleveref}
\usepackage{setspace}
\usepackage{enumerate}
\usepackage{amscd}
\usepackage{tikz-cd}
\usepackage{footmisc}

\newcommand{\mS}{\mathcal{S}}
\newcommand{\mT}{\mathcal{T}}

\newcommand{\M}{\mathbb{M}}
\newcommand{\ot}{\otimes}
\newcommand{\ra}{\rightarrow}
\newcommand{\C}{\mathbb{C}}

\newcommand{\N}{\mathbb{N}}

\newcommand{\mC}{\mathcal{C}}

\newcommand{\mO}{\mathcal{O}}

\newcommand{\ds}{\displaystyle}
\newcommand{\ltp}{\displaystyle\otimes_{\lambda}}
\newcommand{\vep}{\varepsilon}

\theoremstyle{thmit}
\newtheorem{theorem}{Theorem}[section]
\newtheorem{definition}[theorem]{Definition}
\newtheorem{proposition}[theorem]{Proposition}
\newtheorem{lemma}[theorem]{Lemma}
\newtheorem{corollary}[theorem]{Corollary}

\theoremstyle{thmrm}

\newtheorem{remarks}[theorem]{Remark}

\title[Matrix ordering and algebraic aspects of $\lambda$-theory]{\textbf{Polynomials in Operator space theory: Matrix ordering and algebraic aspects}}

\author[P. Luthra]{Preeti Luthra}
\address{Department of Mathematics\\ University of Delhi\\ Delhi-110007, INDIA}
\email{maths.preeti@gmail.com}

\author[A. Kumar]{Ajay Kumar$^*$}
\address{Department of Mathematics\\ University of Delhi\\ Delhi-110007, INDIA}
\email{akumar@maths.du.ac.in}

\author[V. Rajpal]{Vandana Rajpal}
\address{Department of Mathematics\\Shivaji College\\
University of Delhi\\
Delhi\\
India.}
\email{vandanarajpal.math@gmail.com}

  
\begin{document}

\begin{abstract} We extend the $\lambda$-theory of operator spaces given in \cite{defantpolynomials}, that generalizes  the notion of the projective, Haagerup and Schur tensor norm for operator spaces to matrix ordered spaces and Banach $*$-algebras. Given matrix regular operator spaces and operator systems, we introduce cones related to $\lambda$ for the algebraic tensor product that respect the matricial structure of matrix regular operator spaces and operator systems, respectively. The ideal structure of $\lambda$-tensor product of $C^*$-algebras has also been discussed.
\end{abstract}

 \maketitle
 \footnotetext[1]{Corresponding author}
 
\section{Introduction}

$C^*$-algebras are rich objects as they come along with matrix norms that are not only uniquely related to algebraic structure but are also known to have matricial cone structures being closely related to those norm. Although, operator spaces and their tensor products are primarily defined in terms of appropriate matrix norms, over the years it has been observed that some operator space tensor products of $C^*$-algebras still possess few algebraic properties that can be characterized in terms of the individual algebras (\cite{allenideal, kumarideal}). Regarding ordering, although operator spaces may possess some order structure unrelated to the matrix norms, it was Schreiner \cite{schreiner} who defined matrix regular operator spaces to be the spaces where there is a relationship between norm and order. In matrix regular operator spaces, there are enough positive elements so that each element can be written as a linear combination of positive elements. Recently introduced tensor product theory for (unital) operator systems category (\cite{KPTT1}) shows that this matrix order-matrix norm relation is successfully carried over.

 Defant and Wiesner in \cite{defantpolynomials} (see also \cite{wiesnerpolynomials}) have given a $\lambda$-theory which generalizes the definitions of the projective, Haagerup and Schur tensor norm for operator spaces. It is thus natural to ask for appropriate matrix ordering and algebraic structure that is compatible with this generalized $\lambda$-theory. In \cite{Itoh2000} and \cite{rajpalschur}, the projective  and Schur operator space tensor product of matrix ordered operator spaces are shown to be matrix ordered respectively. Further, Han in \cite{hanpredual} successfully introduced cones at each matrix level of the tensor product of operator spaces that are closely related to projective and injective operator space tensor norms thereby, constructing two extremal tensor products of matrix regular operator space.

\Cref{s:1} discusses the prerequisites. Next, we introduce conditions (O1)-(O3) in \Cref{s:2} that enables generalization of Han's (\cite{hanpredual}) operator space tensor product matrix regularity results to $\lambda$-theory of operator spaces. In \Cref{s:3}, we show that the cones defined in \Cref{s:2} also preserve the operator system structure. Finally in \Cref{s:4}, we show that the techniques to study ideal structure of operator space tensor product of $C^*$-algebras can be extended to $\lambda$-theory.

\section{Preliminaries}\label{s:1}
\subsection{The $\lambda$-theory \cite{defantpolynomials}\cite{wiesnerpolynomials}} 

Let $V_1,V_2,\ldots,V_m;W$ be operator spaces and let $\phi$ be an $m$-linear mapping on $V_1 \times V_2 \times \cdots \times V_m$ into $W$. Given a sequence of matrix products $\lambda=(\lambda_k),$ for each $k$, $\lambda_k$ is an $m$-linear mapping: $$\lambda_k: M_k\times \cdots \times M_k \ra M_{\tau(k)},$$ where $\tau(k) \in \N$ is a natural number only depending on $k$, tensorizing $\lambda_k$ with $\phi$ leads to the $m$-linear mapping 
\begin{align*} \phi_{\lambda_k}:= \lambda_k \ot \phi  : M_k(V_1) \ot \cdots \ot M_k(V_m) \ra M_{\tau(k)}(W),\notag \\ (\alpha_1 \ot v_1, \cdots ,  \alpha_m \ot v_m) \mapsto \lambda_k(\alpha_1,\cdots, \alpha_m) \ot \phi(v_1,v_2,\cdots, v_m). \end{align*} Further $$\|\phi\|_{cb,\lambda}:=\sup_{k \in \N}\{\|\phi_{\lambda_k}(x_1,\ldots,x_m)\|_{M_{\tau(k)}(W)} \; : \;\|x_i\|_{M_k(V_i)} \leq 1\}$$
and $$CB_{\lambda}(V_1,\ldots,V_m;W):=\{\phi \in L(V_1,\ldots,V_m;W)\;:\;\|\phi\|_{cb,\lambda} < \infty\}.$$

Since $m$-fold tensor product on $V_1 \times V_2 \times \cdots \times V_m$ is an $m$-linear map onto $\ot_{i=1}^m V_i$, the natural map obtained as above by tensorizing with $\lambda_k$ is represented by $\otimes_{\lambda_k}$:
\begin{align*} \ot_{\lambda_k}: M_k(V_1) \ot \cdots \ot M_k(V_m) \ra M_{\tau(k)}(\ot_{i=1}^m V_i),\\(\alpha_1 \ot v_1, \cdots, \alpha_m \ot v_m) \mapsto \lambda_k(\alpha_1,\cdots, \alpha_m) \ot v_1 \ot v_2 \ot \cdots \ot v_m.\end{align*}

In \cite{defantpolynomials, wiesnerpolynomials}, a tensor norm $\lambda$ was defined as:
\begin{equation}\label{lnorm}
\|u\|_{\lambda,k}=\inf\{\|\alpha\|\|v_1\|\|v_2\|\cdots\|v_m\|\|\beta\|\}
\end{equation}
for any element $u \in M_k(\ds\ot_{i=1}^mV_i)$, where  the infimum is taken over arbitrary decompositions $u= \alpha \otimes_{\lambda_j} (v_1,v_2,\cdots, v_m) \beta$, $\alpha\in M_{k, \tau(j)}$, $\beta\in M_{\tau(j), k}$, $v_t\in M_j(V_t)$.

Keeping the notations from \cite{wiesnerpolynomials, defantpolynomials} unchanged, e.g. $\vep_{i,j}:=\vep^{[k,l]}_{i,j}\in M_{k,l}$ denotes the matrix which is 1 in the $(i,j)$-th coordinate and zero elsewhere, $\vep^{[k]}_{i,j}:=\vep^{[k,k]}_{i,j},\; \vep_i:=\vep_{i,i}, \;\vep^{[k]}_i:=\vep^{[k,k]}_{i,i}$ and $\vep^{[k,l]}_{i,j}=0$ if $(i,j) \notin \{1,\ldots,k\} \times \{1,\ldots,l\}$, we state the three technical conditions (E1)-(E3) that were isolated on the family $\lambda=(\lambda_n)_{n \in \N}$ to assure that the $\|\cdot\|_{\lambda,k}$, generates an operator space structure on $V_1 \ot V_2 \ot \cdots \ot V_m$ \cite[Proposition 4.1]{defantpolynomials}:

\begin{itemize}
\item[(E1)] For all $k \in \mathbb{N}$ there exist $p \in \N$ and matrices $S \in M_{k,\tau (p)}$ , $T \in M_{\tau (p),k}$, $a_1 , \cdots , a_k \in M_p$ such that for all $ j_1 , \cdots, j_m
 \in \{1, \cdots, k\}$:
$$
S \lambda_p(a_{j_1}, \cdots,a_{j_m}) T= \left\{
        \begin{array}{ll}
         \vep_j^{[k]} & \qquad \text{if}\;\; j_1=j_2=\cdots=j_m=j, \\
          0 & \qquad \text{otherwise.}
          \end{array}
    \right.
$$ 
\item[(E2)] For all $ r, s \in  \mathbb{N}$ there exist matrices  $P \in M_{\tau (r)+\tau (s), \tau(r+s)}$, with $\|P\|\leq 1$ such that for all $ (i_k , j_k ) \in 
 \{1, \cdots , r\}^2 \cup \{r + 1, \cdots , r + s\}^2$ with $1\leq k\leq m$:
 
\begin{align}
\begin{split}
& P \lambda_{r+s}( \vep_{i_1, j_1}^{[r+s]} , \ldots, \vep_{i_m ,j_m }^{[r+s]})P^* \\
&= \mathrm{diag}\Big( \lambda_r (\vep_{i_1 ,j_1}^{[r]} , \cdots , \vep_{i_m,j_m}^{[r]}), \lambda_{s}(\vep_{i_1-r ,j_1-r}^{[s]} , \cdots , \vep_{i_m-r,j_m-r}^{[s]})\Big).\notag
\end{split}
\end{align}
\item[(E3)] $\lambda_1(1,1,\cdots, 1)=1$ and $\ds\sup_{k\in \N}\|\lambda_k\| <\infty$.
\end{itemize}

\vspace{3mm}

If in addition $\lambda$ satisfies:

(N1)$ \qquad \tau(1) =1\qquad \qquad \text{and} \qquad \qquad$(N2)$ \quad \|\lambda_j\|=1$ for all $j \in \N,$

then $\ltp \C =\C$ completely isometric \cite[Proposition 4.13]{wiesnerpolynomials}.

\vspace{3mm}

For $j \in \{1,\ldots,m\}$, if $\lambda$ further satisfy conditions:

\begin{itemize}\label{w2}

\item[(W1)] For all $\gamma \in M_p$ there exists matrices $P \in M_{p,\tau(p)}$, $Q \in M_{\tau(p), p}$ with $\|P\|, \|Q\| \leq 1$ such that

\begin{align*}
\gamma = P \lambda_p(I_p,\cdots,I_p,\underbrace{\gamma}_\text{j-th position}, I_p,\cdots, I_p)Q
\end{align*}

\item[(W2)] For all $\alpha_1,\cdots,\alpha_m \in M_p$, $\beta_1,\cdots, \beta_m \in M_q$ there exist matrices $S \in M_{\tau(p)\tau(q),\tau(pq)}, T \in M_{\tau(pq),\tau(p)\tau(q)}$ with $\|S\|,\|T\| \leq 1$ such that

\begin{align*}
\lambda_p(\alpha_1,\cdots,\alpha_m) \ot \lambda_q(\beta_1,\cdots,\beta_m)=S\lambda_{pq}(\alpha_1\ot\beta_1,\cdots,\alpha_m \ot \beta_m) T
\end{align*}

\end{itemize}

then the mapping 

\begin{align*}
\Phi^{(j)} : (V_1 \ot \cdots \ot M_p(V_j) \ot \cdots \ot V_m , \|\cdot\|_{\lambda}) \ra M_p(\ltp V_i) \\ v_1 \ot \cdots\ot(\alpha \ot v_j) \ot \cdots v_m \mapsto \alpha \ot (v_1 \ot \cdots \ot v_m)
\end{align*}

is completely contractive \cite[Proposition 12.2]{wiesnerpolynomials}.

\vspace{3mm}

If $\lambda$ satisfies (N1)-(N2), (E1)-(E3) and (W1)-(W2) then $\ot^\lambda V_i$, the completion of  $\ot_\lambda V_i$ with respect to $\|\cdot\|_{\lambda}$ norm, is an operator space tensor product denoted by $\lambda$-operator space tensor product in the sense of \cite{bptensor}.

\vspace{3mm}

The Kronecker product, matrix product and mixed product fulfill all the above conditions.

\vspace{3mm}

We assume throughout that $\lambda$ satisfies all the prescribed conditions.

\vspace{3mm}
%
%
%
%
%

\subsection{Matrix regular operator space and operator systems}
An operator space $V$ is called a matrix ordered operator space if: 
\begin{enumerate}
\item $(V,\{M_n(V)^+\}_{n=1}^{\infty})$ is a matrix ordered vector space i.e. for each $n \in \N$, $M_n(V)$ is a $*$-ordered vector space with cone $M_n(V)^+$ and $A \in \M_{n,m}$ implies $A^*M_n(V)^+A \subseteq M_m(V)^+$.
\item the $*$-operation is an isometry on $M_n(V)$.
\item the cones $M_n(V)^+$ are closed.
\end{enumerate}

   A matrix ordered
operator space $V$ is called matrix regular \cite[Definition 3.1.9]{schreiner} if for each $n \in \N$ and
for all  $v\in M_n(V)_{sa}$, the following conditions hold : 

\begin{enumerate}

\item $u\in M_n(V)^+$ and $-u \leq v \leq u$ implies that $\|v\|_n \leq \|u\|_n $.

\item $\|v\|_n \leq 1$ implies that there exists $u\in M_n(V)^+$ such that $\|u\|_n \leq 1$ and $-u \leq v \leq u$.

\end{enumerate}

Next result from \cite{schreiner} giving a necessary and sufficient for a matrix ordered operator space $V$ to be matrix regular is quite useful:

\begin{theorem}{\cite[Theorem 3.4]{schreiner}} A matrix ordered operator space V is matrix regular if and only if the following condition holds: for all
$x \in  M_n(V)$,  $\|x\|_n < 1$ if and only if there exist $a, d \in M_n(V)^+$, $\|a\|_n <1$ and $\|d\|_n < 1$, such that $\begin{pmatrix}
a & x \\ x^* & d
\end{pmatrix} \in M_{2n}(V)^+$.
\end{theorem}

The positive cone of a matrix regular operator
space is always proper.

Adopting the methodology of \cite{hanpredual}, the norms on matrix regular operator spaces are not assumed to be complete.

For a matrix ordered operator space $V$ and its dual space $V^*,$ the positive cone on $M_n(V^*)$ for each $n \in \N$ is defined by $M_n(V^*)^+= CB(V,M_n) \cap CP(V,M_n)$.
The operator space dual $V^*$ with this positive cone is a matrix ordered operator space \cite[Corollary 3.2]{schreiner}.

An (abstract) operator system (\cite[Definition 2.2]{KPTT1})
is a triple $(V ,\{\mC_n\}_{n=1}^{\infty},e$), where $V$ is a complex $*$-vector space, $\{\mC_n\}_{n=1}^{\infty}$ is a matrix ordering on $V,$ and $e \in V_{sa}$ is an Archimedean matrix order unit, i.e. for all $v \in M_n(V)_{sa}$, \begin{enumerate}
\item  there exists a real number $r >0$ such that $re_n >v$ and
\item for each $n \in \N$ and $e_n= \begin{pmatrix}
e & & \\
  & \ddots & \\
  & & e
\end{pmatrix}$, $s e_n + v \in \mC_n$ for all $s >0$ implies $v \in \mC_n.$
\end{enumerate}

\section{$\lambda$-theory and Matrix regularity}\label{s:2}
In this section, we provide three additional conditions on $\lambda=(\lambda_n)_{n\in \N}$ to introduce an order structure to $\lambda$-theory that preserves matrix regularity. Further, using our conditions (O1)-(O3) defined below, we prove that the results of \cite{hanpredual} hold true in a more general setting introduced by \cite{defantpolynomials,wiesnerpolynomials}.

\vspace{3mm}

For a sequence $\lambda=(\lambda_n)_{n\in\N}$ of $m$-linear mappings $\lambda_k \in L(^mM_k;M_{\tau(k)})$ consider the following three properties:

\vspace{3mm}

(O1)\label{o1} For each $r \in \N$, $$\lambda_r(\vep^{[r]}_{i_1,j_1},\vep^{[r]}_{i_2,j_2},\cdots,\vep^{[r]}_{i_m,j_m}) = \lambda_r(\vep^{[r]}_{j_1,i_1},\vep^{[r]}_{j_2,i_2},\cdots,\vep^{[r]}_{j_m,i_m})^* \in M_{\tau(r)},$$ for all  $(i_k,j_k) \in \{1,\ldots,r\} \times \{1,\ldots,r\}$, and $k=1,2,\dots m$.

\vspace{3mm}

\vspace{3mm}

{(O2)}\label{o2} For $r \in \mathbb{N}$, the permutation matrix $P \in M_{2\tau(r),\tau(2r)}$ with
$\|P\| \leq 1$ obtained in (E2) and $(i_k,j_k)\in R \cup S$, where $R:=  \{1,\cdots,r\}\times \{r+1,r+2,\cdots, 2r\}$
and $S:= \{r+1,r+2,\cdots, 2r\} \times  \{1,\cdots,r\}$ 

\begin{align}
\begin{split}
& P\lambda_{2r} (\vep_{i_1,j_1}^{[2r]}, \cdots, \vep_{i_m,j_m}^{[2r]}) P^* \\
& = \mathrm{adiag} \big(\lambda_r(\vep_{i_1,j_1-r}^{[r]},\cdots,\vep_{i_1,j_1-r}^{[r]}),\lambda_r(\vep_{i_1-r,j_1}^{[r]}, \cdots,
 \vep_{i_m-r,j_m}^{[r]})\big); \notag
 \end{split}
 \end{align}
 adiag being an anti-diagonal matrix, where all the entries are zero except those on the diagonal going from the upper right corner to the lower left corner. 
 
 \vspace{3mm}

{(O3)}\label{o3} For each $r \in \N$, the map
 $$\displaystyle \overset{p_1\ldots p_m}{\ot_{\lambda_r}}= \displaystyle \overset{p_1\ldots p_m}{\ot} \ot \lambda_r  : M_r(M_{p_1}) \ot M_r(M_{p_2}) \ot \ldots \ot M_r(M_{p_m}) \ra M_{\tau(r)}(M_{p_1p_2\ldots p_m})$$
obtained by tensorizing $\lambda_r$ with the Kronecker product on matrix algebras $M_{p_1},\ldots, M_{p_m}$ $$\displaystyle \overset{p_1\ldots p_m}{\ot} : M_{p_1} \times M_{p_2} \times \ldots \times M_{p_m} \ra M_{p_1p_2\ldots p_m}$$ $$( \alpha_1, \alpha_2, \ldots,  \alpha_m) \mapsto \alpha_1 \ot \alpha_2 \ot \ldots \ot \alpha_m,$$ is positive for all $p_i \in \N$ $(i=1,2,\ldots,m)$. Thus,  $$ \displaystyle \overset{p_1\ldots p_m}{\ot_{\lambda_r}}(\alpha_1 \ot \beta_1, \ldots, \alpha_m \ot \beta_m)= \lambda_r(\beta_1 \ot \ldots \ot  \beta_m) \ot  \alpha_1 \ot \ldots \ot \alpha_m \in M_{\tau(r)p_1p_2\ldots p_m}^+, $$ whenever $\alpha_i \ot \beta_i \in (M_{p_i} \ot M_r)^+$, $p_i \in \N$ $(i=1,2, \ldots m).$

\vspace{3mm}

Recall from \cite[Proposition 4.1]{schreiner} (see also \cite[Proposition 4.2]{wiesnerpolynomials}), given a sequence $\lambda=(\lambda_n)$ of $m$-linear maps and operator spaces $V_1,V_2,\ldots,V_m$ any element $u \in M_k(V_1 \ot \ldots V_m)$ has a representation $u=\alpha \ot_{\lambda_r} (v_1, v_2, \ldots, v_m) \beta$ where $\alpha \in M_{n, \tau(r)}, \beta \in M_{\tau(r),n}, v_i \in M_{r}(V_i), r \in \N.$ 

Next, we analyze the above conditions in view of their applications to matrix ordered spaces:

 \begin{lemma}\label{O1-3} Let $\lambda=(\lambda_n)$ be sequence of $m$-linear maps and $V_1,\ldots V_m$ be matrix ordered operator spaces. For any  $\alpha \ot_{\lambda_r}(v^{(1)},v^{(2)},\ldots,v^{(m)})\beta \in M_n(\ot_{\lambda} V_i)$; $\alpha \in M_{n, \tau(r)}, \beta \in M_{\tau(r),n}, v^{(i)} \in M_{r}(V_i), i=1,\ldots,m, r\in \N,$ we have:
 \begin{enumerate}[(i)]
 \item If $\lambda$ satisfies \emph{(O1)}, then $*$-map defined as $$\big(\alpha \ot_{\lambda_r}(v^{(1)},v^{(2)},\ldots,v^{(m)})\beta \big)^*= \beta^* \ot_{\lambda_r}\big((v^{(1)})^*,(v^{(2)})^*,\ldots,(v^{(m)})^*\big) \alpha^*,$$ is a well defined involution.
 \item If $\lambda$ satisfies \emph{(O2)}, then for $u^{(i)},\tilde{u}^{(i)} \in M_{r}(V_i)$, $i=1,\ldots,m$
\begin{align}
\begin{split}
& \begin{pmatrix}
\ot_{\lambda_r}(u^{(1)},u^{(2)},\ldots,u^{(m)}) &  \ot_{\lambda_r}(v^{(1)},v^{(2)},\ldots,v^{(m)}) \\
 \ot_{\lambda_r}\big((v^{(1)})^*,(v^{(2)})^*,\ldots,(v^{(m)})^*\big) &  \ot_{\lambda_r}(\tilde{u}^{(1)},\tilde{u}^{(2)},\ldots,\tilde{u}^{(m)})  
\end{pmatrix} \\
& = P\ot_{\lambda_{2r}}
\Bigg(\begin{pmatrix}
u^{(1)}  & v^{(1)}\\
(v^{(1)})^*  &   \tilde{u}^{(1)}
\end{pmatrix} ,\ldots, 
\begin{pmatrix}
u^{(m)}  & v^{(m)}\\
(v^{(m)})^*  &   \tilde{u}^{(m)}
\end{pmatrix}  \Bigg)P^*  . \notag 
\end{split}
\end{align}
 \item For $\lambda$ and $\mu$, if $(\mu_p)_{\lambda_r}$ is a positive map $(p,r \in \N)$, then for  $v^{(i)} \in M_r(V_i)^+$ and completely positive maps $\phi^{(i)} : V_i \ra M_{p_i}$, $i=1,\ldots,m$, we have
\begin{align*}\big(\ot_{\mu_p} (\phi^{(1)},\ldots,\phi^{(m)})\big)_{n}(\alpha \ot_{\lambda_r}(v^{(1)},v^{(2)},\ldots,v^{(m)})\alpha^*) \in M_{np_1\ldots p_m}^+.
\end{align*}
In particular, if $\lambda$ satisfies \emph{(O3)},
\begin{align*}\big( (\phi^{(1)} \ot \ldots \ot \phi^{(m)})\big)_{n}(\alpha \ot_{\lambda_r}(v^{(1)},v^{(2)},\ldots,v^{(m)})\alpha^*) \in M_{np_1\ldots p_m}^+.
    \end{align*}
 \end{enumerate}
 \end{lemma}
 \begin{proof} To obtain (i) one can easily verify that the $*$-operation is conjugate linear and involutive.
 \begin{enumerate}
 \item[(ii)] We have
  
\begin{align}
 \begin{split} & \begin{pmatrix}
\ot_{\lambda_r}(u^{(1)},u^{(2)},\ldots,u^{(m)}) &  \ot_{\lambda_r}(v^{(1)},v^{(2)},\ldots,v^{(m)}) \\
 \ot_{\lambda_r}((v^{(1)})^*,(v^{(2)})^*,\ldots,(v^{(m)})^*) &  \ot_{\lambda_r}(\tilde{u}^{(1)},\tilde{u}^{(2)},\ldots,\tilde{u}^{(m)})  
\end{pmatrix}\\ &  = \mathrm{diag}\big(\ot_{\lambda_r}(u^{(1)},u^{(2)},\ldots,u^{(m)}) , \ot_{\lambda_r}(\tilde{u}^{(1)},\tilde{u}^{(2)},\ldots,\tilde{u}^{(m)}\big) + \\& \qquad \mathrm{adiag}\big(\ot_{\lambda_r}(v^{(1)},v^{(2)},\ldots,v^{(m)}) , \ot_{\lambda_r}((v^{(1)})^*,(v^{(2)})^*,\ldots,(v^{(m)})^*)\big). \notag
\end{split}
\end{align}

Setting $R_1:= \{1,2,\dots,r\}^2$, $R_2:=\{1,2,\dots,r\} \times \{r+1,r+2,\ldots,2r\}$,\\ $R_3:=\{r+1,r+2,\ldots,2r\} \times \{1,2,\ldots,r\}$ and $R_4:=\{r+1,r+2,\ldots,2r\}^2 $,
let $u^{(t)}:= \ds \sum_{(k_t,l_t)\in R_1} \vep^{[r]}_{k_t,l_t} \otimes u_{k_t,l_t}^{(t)}$, $v^{(t)}:= \ds\sum_{(k_t,l_t)\in R_2} \vep^{[r]}_{k_t,l_t-r} \otimes v_{k_t,l_t}^{(t)}$,\\ $(v^{(t)})^* =  \ds\sum_{(k_t,l_t)\in R_3} \vep^{[r]}_{k_t-r,l_t} \otimes (v_{k_t,l_t}^{(t)})^*$ and  $\tilde{u}^{(t)}:= \ds\sum_{(k_t,l_t)\in R_4} \vep^{[r]}_{k_t - r,l_t} \otimes \tilde{u}_{k_t,l_t}^{(t)}$.\\Define $x^{(t)}: = \mathrm{diag} (u^{(t)}, \tilde{u}^{(t)})$ and $y^{(t)} := \mathrm{adiag} (v^{(t)}, (v^{(t)})^*)$, so that 
 $$ x^{(t)}_{k,l}= \left\{
        \begin{array}{ll}
         u^{(t)}_{k,l} & \qquad \text{if}\;\; (k,l)  \in R_1, \\
           \tilde{u}^{(t)}_{k,l} & \qquad \text{if}\;\;  (k,l) \in R_4
          \end{array}
    \right\} \:\; \text{and} \;\; y^{(t)}_{k,l}= \left\{
        \begin{array}{ll}
         v^{(t)}_{k,l} & \qquad \text{if}\;\; (k,l)  \in R_2, \\
           (v^{(t)})^*_{k,l} & \qquad \text{if}\;\;  (k,l)  \in R_3
          \end{array} \right\} .
$$ 
Then,

\begin{align}
\begin{split}
& \mathrm{adiag}\big(\ot_{\lambda_r}(v^{(1)},v^{(2)},\ldots,v^{(m)}) , 0) \\
& = \mathrm{adiag}\big( \sum_{(k_m,l_m) \in R_2} \cdots \sum_{(k_1,l_1) \in R_2} \lambda_r(\vep^{[r]}_{k_1,l_1-r},\cdots,\vep^{[r]}_{k_m,l_m-r}) \ot v_{k_1,l_1}^{(1)} \ot \cdots   \ot v_{k_m,l_m}^{(m)}, 0\big)\\
& =  \sum_{(k_m,l_m) \in R_2} \cdots \sum_{(k_1,l_1) \in R_2} \mathrm{adiag}\big(\lambda_r(\vep^{[r]}_{k_1,l_1-r},\cdots,\vep^{[r]}_{k_m,l_m-r}),0\big) \ot y_{k_1,l_1}^{(1)} \ot \cdots   \ot y_{k_m,l_m}^{(m)} \\
& \overset{\text{\textbf{(O2)}}}{=}  \sum_{(k_m,l_m) \in R_2} \cdots \sum_{(k_1,l_1) \in R_2}  P\lambda_{2r} (\vep_{k_1,l_1}^{[2r]}, \cdots, \vep_{k_m,l_m}^{[2r]}) P^* \ot y_{k_1,l_1}^{(1)} \ot \cdots   \ot y_{k_m,l_m}^{(m)} \notag
\end{split}
\end{align}

Also,

\begin{align}
\begin{split}
& \mathrm{adiag}\big(0,\ot_{\lambda_r}((v^{(1)})^*,(v^{(2)})^*,\ldots,(v^{(m)})^*)) \\
& = \mathrm{adiag}\big( 0,\sum_{(k_m,l_m) \in R_3} \cdots \sum_{(k_1,l_1) \in R_3} \lambda_r(\vep^{[r]}_{k_1-r,l_1},\cdots,\vep^{[r]}_{k_m-r,l_m}) \ot (v_{k_1,l_1}^{(1)})^* \ot \cdots   \ot (v_{k_m,l_m}^{(m)})^*\big)\\
& =  \sum_{(k_m,l_m) \in R_3} \cdots \sum_{(k_1,l_1) \in R_3} \mathrm{adiag}\big(0, \lambda_r(\vep^{[r]}_{k_1-r,l_1},\cdots,\vep^{[r]}_{k_m-r,l_m}\big) \ot y_{k_1,l_1}^{(1)} \ot \cdots   \ot y_{k_m,l_m}^{(m)} \\
& \overset{\text{\textbf{(O2)}}}{=} \sum_{(k_m,l_m) \in R_3} \cdots \sum_{(k_1,l_1) \in R_3} P\lambda_{2r} (\vep_{k_1,l_1}^{[2r]}, \cdots, \vep_{k_m,l_m}^{[2r]}) P^* \ot y_{k_1,l_1}^{(1)} \ot \cdots   \ot y_{k_m,l_m}^{(m)} \notag
\end{split}
\end{align}

Thus,

\begin{align}
\begin{split}
 & \mathrm{adiag}\big(\ot_{\lambda_r}(v^{(1)},v^{(2)},\ldots,v^{(m)}) , \ot_{\lambda_r}((v^{(1)})^*,(v^{(2)})^*,\ldots,(v^{(m)})^*)\big)\\
 & =\sum_{(k_m,l_m) \in R_2 \cup R_3} \cdots \sum_{(k_1,l_1) \in  R_2 \cup R_3} P\lambda_{2r} (\vep_{k_1,l_1}^{[2r]}, \cdots, \vep_{k_m,l_m}^{[2r]}) P^*  \ot y_{k_1,l_1}^{(1)} \ot \cdots   \ot y_{k_m,l_m}^{(m)} \notag
 \end{split}
\end{align}

Similarly, using (E2),

\begin{align}
\begin{split}
& \mathrm{diag}\big(\ot_{\lambda_r}(u^{(1)},u^{(2)},\ldots,u^{(m)}) , \ot_{\lambda_r}(\tilde{u}^{(1)},\tilde{u}^{(2)},\ldots,\tilde{u}^{(m)}\big)\\ &=\sum_{(k_m,l_m) \in R_1 \cup R_4} \cdots \sum_{(k_1,l_1) \in  R_1 \cup R_4} P\lambda_{2r} (\vep_{k_1,l_1}^{[2r]}, \cdots, \vep_{k_m,l_m}^{[2r]}) P^*  \ot x_{k_1,l_1}^{(1)} \ot \cdots   \ot x_{k_m,l_m}^{(m)}\notag
\end{split}
\end{align} Therefore,
\begin{align}
 \begin{split} 
 & \begin{pmatrix}
\ot_{\lambda_r}(u^{(1)},u^{(2)},\ldots,u^{(m)}) &  \ot_{\lambda_r}(v^{(1)},v^{(2)},\ldots,v^{(m)}) \\
 \ot_{\lambda_r}((v^{(1)})^*,(v^{(2)})^*,\ldots,(v^{(m)})^*) &  \ot_{\lambda_r}(\tilde{u}^{(1)},\tilde{u}^{(2)},\ldots,\tilde{u}^{(m})  
\end{pmatrix}\\ 
& = P \Big(\sum_{(k_m,l_m) \in \cup_{i=1}^{4} R_i} \cdots \sum_{(k_1,l_1) \in  \cup_{i=1}^{4} R_i} \ot_{\lambda_{2r}}\big( \vep_{k_1,l_1}^{[2r]} \ot x_{k_1,l_1}^{(1)} + \vep_{k_1,l_1}^{[2r]} \ot y_{k_1,l_1}^{(1)},\cdots \\
& \qquad \qquad \qquad\qquad\qquad\qquad\qquad\qquad\qquad \cdots,\vep_{k_m,l_m}^{[2r]} \ot x_{k_m,l_m}^{(m)} + \vep_{k_m,l_m}^{[2r]} \ot y_{k_m,l_m}^{(m)}\big)\Big)P^*\\
 & = P\ot_{\lambda_{2r}}
\Bigg(\begin{pmatrix}
u^{(1)}  & v^{(1)}\\
(v^{(1)})^*  &   \tilde{u}^{(1)}
\end{pmatrix} ,\ldots, 
\begin{pmatrix}
u^{(m)}  & v^{(m)}\\
(v^{(m)})^*  &   \tilde{u}^{(m)}
\end{pmatrix}  \Bigg)P^*.  \notag
\end{split}
\end{align}

\item[(iii)] Note that, if $v^{(t)}:= \ds\sum_{(k_t,l_t)\in R} \vep^{[r]}_{k_t,l_t} \otimes v_{k_t,l_t}^{(t)};$ $R:=\{1,\ldots,r\}^2$, $t=1,2,\ldots,m$ then,
 
\begin{align}
 \begin{split}
 & (\ot_{\mu_p} ( \phi^{(1)},\ldots,\phi^{(m)})_{\tau_{\mu}(p)}( \ot_{\lambda_r}(v^{(1)},v^{(2)},\cdots,v^{(m)}))\\
 & = \sum_{k_m,l_m}\ldots \sum_{k_1,l_1}(\ot_{\mu_p} ( \phi^{(1)},\ldots,\phi^{(m)})_{\tau_{\mu}(p)}\big(\lambda_r(\vep_{k_1,l_1},\ldots,\vep_{k_m,l_m}) \ot v^{(1)}_{k_1,l_1} \ot \ldots \ot v^{(m)}_{k_m,l_m}\big)\\
 & = \sum_{k_m,l_m}\ldots \sum_{k_1,l_1}\lambda_r(\vep_{k_1,l_1},\ldots,\vep_{k_m,l_m})\ot \big(\ot_{\mu_p} ( \phi^{(1)},\ldots,\phi^{(m)})(v^{(1)}_{k_1,l_1} \ot \ldots \ot v^{(m)}_{k_m,l_m})\big)\\
 &= \sum_{k_m,l_m}\ldots \sum_{k_1,l_1}\lambda_r(\vep_{k_1,l_1},\ldots,\vep_{k_m,l_m}) \ot \mu_p\big(\phi^{(1)}(v^{(1)}_{k_1,l_1}) \ot \ldots \ot \phi^{(m)}( v^{(m)}_{k_m,l_m})\big)\\
 & = \sum_{k_m,l_m}\ldots \sum_{k_1,l_1} (\mu_p)_{\lambda_r} \big(\vep_{k_1,l_1} \ot \phi^{(1)}(v^{(1)}_{k_1,l_1}), \ldots, \vep_{k_m,l_m} \ot \phi^{(m)}(v^{(m)}_{k_m,l_m})\big)\\
 & = \sum_{k_m,l_m}\ldots \sum_{k_1,l_1} (\mu_p)_{\lambda_r} \big(\phi^{(1)}_r(\vep_{k_1,l_1} \ot v^{(1)}_{k_1,l_1}), \ldots, \phi^{(m)}_r(\vep_{k_m,l_m} \ot v^{(m)}_{k_m,l_m}))\big)\\
 & = (\mu_p)_{\lambda_r} \big( \phi^{(1)}_r(v^{(1)}),\ldots,\phi^{(m)}_r(v^{(m)})\big)\\
 & \in M_{\tau(r)p_1p_2\ldots p_m}^+ \nonumber
 \end{split}
\end{align} so that, \begin{align}\begin{split}
  &  (\ot_{\mu_p} ( \phi^{(1)},\ldots,\phi^{(m)})_n(\alpha \ot_{\lambda_r}(v^{(1)},\ldots,v^{(m)})\alpha^*))\\
  & = \alpha \big((\ot_{\mu(p)} ( \phi^{(1)},\ldots,\phi^{(m)}))_{\tau_{\mu}(p)} ( \ot_{\lambda_j}(v_1,v_2,\cdots,v_m))\big)\alpha^* \in M_{nk_1k_2\ldots k_m}^+. \notag
  \end{split} \end{align} If $\lambda$ satisfies (O3), $\mu_p =\displaystyle \overset{p_1\ldots p_m}{\ot}$ gives the desired result.
\end{enumerate} \end{proof}

\vspace{3mm}

\paragraph{\textbf{Verification of Properties (O1)-(O3):}}
\begin{itemize}
\item \emph{Kronecker product:} Property (O1) reduces to  $$ \vep^{[r]}_{i_1,j_1}\ot\vep^{[r]}_{i_2,j_2} \ot \cdots \ot \vep^{[r]}_{i_m,j_m} = ( \vep^{[r]}_{j_1,i_1} \ot \vep^{[r]}_{j_2,i_2} \ot \cdots \ot \vep^{[r]}_{j_m,i_m})^*,$$ which is true.\\
To check for the condition (O3), recall that Kronecker product of two positive matrices is positive, but Kronecker product does not commute, in fact for any square matrices $A$ and $B$, there exists a permutation matrix $S$ such that $B \ot A = S (A \ot B) S^*$. Therefore, for some suitable permutation matrix $S$ we have:

 \begin{align}
\begin{split}
\displaystyle \overset{p_1\ldots p_m}{\ot_{\ot_r}}(\alpha_1 \ot \beta_1, \ldots, \alpha_m \ot \beta_m) & = (\beta_1 \ot \ldots \ot  \beta_m) \ot  \alpha_1 \ot \ldots \ot \alpha_m \\ & = S \big((\alpha_1 \ot \beta_1) \ot \cdots \ot (\alpha_m \ot \beta_m)\big)S^* \\ & \in M_{r^mp_1p_2\ldots p_m}^+, \notag \end{split}
\end{align} whenever $\alpha_i \ot \beta_i \in M_r(M_{p_i})^+$, $p_i \in \N$ $(i=1,2, \ldots m).$ \\
In order to verify (O2), we use the same notations as in proof of \cite[Proposition 4.2]{defantpolynomials}, let $\Delta : M_1 \ra M_{1,m}$, $x \mapsto (x,x,\ldots,x)$ and set $$P_1: = \sum_{r \in \{1,2,\ldots,r\}^m} \vep_{p,p}^{[\Delta r, \Delta 2r]} \qquad \text{and} \qquad  P_2: = \sum_{r \in \{r+1,r+2,\ldots,2r\}^m} \vep_{p- \Delta r,p}^{[\Delta r, \Delta 2r]},$$ and let $P=\begin{pmatrix}
P_1\\P_2
\end{pmatrix}$ 

\begin{align}\label{b6}
\begin{split}
& P\ot_{2r} (\vep^{[2r]}_{i_1,j_1},\cdots, \vep_{i_m,j_m}^{[2r]}) P^* \\
& = P \vep_{(i_1,\ldots,i_m),(j_1,\ldots,j_m)}^{[\Delta 2r, \Delta 2r]} P^*\\
&=\left(
                                                                              \begin{array}{cc}
                                                                                \vep_{(i_1,\cdots,i_m),(j_1,j_2,\cdots,j_m)}^{[\Delta r
,\Delta r]} &  \vep_{(i_1,\cdots,i_m),(j_1,j_2,\cdots,j_m)- \Delta r}^{[\Delta r
,\Delta r]}\\                                                                       \vep_{(i_1,\cdots,i_m)- \Delta r ,(j_1,j_2,\cdots,j_m)}^{[\Delta r
,\Delta r]}        & \vep_{(i_1,\cdots,i_m)-\Delta r,(j_1,j_2,\cdots,j_m)-\Delta r}^{[\Delta r
,\Delta r]} \\
                                                                              \end{array}
                                                                            \right)\\
&= 
\left\{
        \begin{array}{ll}
         \mathrm{adiag}\big( \vep_{(i_1,\cdots,i_m),(j_1,j_2,\cdots,j_m)- \Delta r}^{[\Delta r
,\Delta r]},0\big) & \qquad \text{if}\;\; (i_k,j_k) \in \{1,\cdots,r\} \times \{r+1,\cdots,2r\}, \\
          \mathrm{adiag}\big( 0, \vep_{(i_1,\cdots,i_m)- \Delta r ,(j_1,j_2,\cdots,j_m)}^{[\Delta r
,\Delta r]} \big) & \qquad \text{if}\;\; (i_k,j_k) \in \{r+1,\cdots,2r\} \times \{1,\cdots,r\}\\ 0 & \qquad \text{else}
          \end{array}
    \right.\notag 
   \end{split}
   \end{align}

   \item \emph{Schur product:} Here property (O1) takes the form $$ \vep^{[r]}_{i_1,j_1}\odot\vep^{[r]}_{i_2,j_2} \odot \cdots \odot \vep^{[r]}_{i_m,j_m} = ( \vep^{[r]}_{j_1,i_1} \odot \vep^{[r]}_{j_2,i_2} \odot \cdots \odot \vep^{[r]}_{j_m,i_m})^*,$$ which is true.\\To check for the condition (O3), recall that Schur product of two positive matrices is positive, for any square matrices  $A,B,C \; \text{and}\; D$ of order $n$,  $(A \odot B) \ot (C \odot D) =(A \ot C )\odot (B \ot D) $ (see \cite[Proposition 10.5]{wiesnerpolynomials}) and there exist a matrix $\mathcal{E} \in M_{n,k}$, $A', B' \in M_{k}$ such that $(A \ot B)=\mathcal{E} (A' \odot B') \mathcal{E}^*$. Therefore we have,
 \begin{align}
\begin{split}
\displaystyle \overset{p_1\ldots p_m}{\ot_{\odot_r}}(\alpha_1 \ot \beta_1, \ldots, \alpha_m \ot \beta_m) & = (\beta_1 \odot \ldots \odot  \beta_m) \ot  \alpha_1 \ot \ldots \ot \alpha_m 
\\ &  \in M_{rp_1p_2\ldots p_m}^+, \notag \end{split}
\end{align} whenever $\alpha_i \ot \beta_i \in M_r(M_{p_i})^+$, $p_i \in \N$ $(i=1,2, \ldots m).$ \\Again as for (E2), for $r \in \mathbb{N}$ and $(i_q,j_q)\in [\{1,\cdots,r\} \times \{r+1,\cdots,2r\}] \cup [\{r+1,\cdots,2r\} \times \{1,\cdots,r\}] $, let $P:= I_{2r}$, we have
\begin{align}
\begin{split}       
      & P \odot_{2r} (\vep_{i_1,j_1}^{[2r]}, \cdots, \vep_{i_m,j_m}^{[2r]})P^*\\
      & =\left(
                                                                                             \begin{array}{cc}
                                                                                              0&  \vep_{i_1,j_1-r}^{[r]}  \\
                                                                                              \vep_{i_1-r,j_1}^{[r]}  & 0 \\
                                                                                             \end{array}
                                                                                           \right)
       \odot \cdots \odot \left(
                         \begin{array}{cc}
                           0 & \vep_{i_m,j_m-r}^{[r]} \\
                            \vep_{i_m-r,j_m}^{[r]} & 0 \\
                         \end{array}
                       \right)\\
       &=  \left(
                         \begin{array}{cc}
                           0 &\vep_{i_1,j_1-r}^{[r]} \odot\cdots \odot \vep_{i_m,j_m-r}^{[r]} \\
                           \vep_{i_1-r,j_1}^{[r]} \odot \cdots \odot  \vep_{i_m-r,j_m}^{[r]} & 0 \\
                         \end{array}
                       \right)  \notag
                       \end{split}
                       \end{align} implying that (O2) holds.
\item \emph{Matrix Product:} One can easily see that this product may not satisfy (O1), (O2) and (O3).
\item \emph{Mixed Product:} One can mix above listed products to construct a new one, for example \cite[Chapter 9]{wiesnerpolynomials}, $\lambda =(\lambda_k)_k$  with
\begin{align*} \lambda_k : M_{k} \times M_k \times M_k \times M_k \ra M_{k^2},\notag \\
(\alpha_1,\alpha_2,\alpha_3,\alpha_4) \mapsto (\alpha_1 \bullet \alpha_2) \ot  (\alpha_3 \bullet \alpha_4) \notag \end{align*} Clearly, it does not satisfy any of the (O1)-(O3) as $\bullet$ does not.

One can similarly talk of $\lambda =(\lambda_k)_k$  with
\begin{align*}\lambda_k : M_{k} \times M_k \times M_k \times M_k \ra M_{k^2},\notag \\(\alpha_1,\alpha_2,\alpha_3,\alpha_4) \mapsto (\alpha_1 \odot \alpha_2) \ot  (\alpha_3 \odot \alpha_4),\end{align*} which clearly satisfies all the conditions (O1)-(O3).    \hspace{25mm}         \qedsymbol 
 \end{itemize}  The self-adjoint elements in $M_n(\ot_\lambda V_i)$ have a special representation:

\begin{proposition}\label{selfadjoint}
Let $V_i$, $1\leq i \leq m$, be  matrix ordered operator spaces and let $\lambda=(\lambda_n)_{n\in \N}$ be a sequence of $m$-multilinear mappings satisfying \emph{(O1)} and \emph{(O2)}. If $u\in M_n(\ot_\lambda V_i)_{sa},$ then $u$ has a representation: $$u=\alpha\ot_{\lambda_j}(x_1,x_2,\cdots,x_m)\alpha^*$$ where $\alpha\in M_{n,\tau(j)}$, $x_t\in M_j(V_t)_{sa}$, $t=1,2,\cdots,m.$ Moreover, 
\begin{align*}
\begin{split}
\|u\|_{\lambda,n} & =\inf\big\{\|\alpha\|^2\|x_1\|\|x_2\|\cdots \|x_m\| :u=\alpha\ot_{\lambda_j}(x_1,x_2,\cdots,x_m)\alpha^*,\\ & \qquad \qquad \qquad \qquad \qquad \qquad \qquad \qquad \alpha\in M_{n,\tau(j)}, x_t\in M_j(V_t)_{sa}, j\in \mathbb{N}\big\}. \notag
\end{split}
\end{align*}
 \end{proposition}
\begin{proof}
Suppose $u\in M_n(\ot_\lambda V_i)_{sa}$. Given $\epsilon>0$, there exist $\alpha\in M_{n,\tau(j)}$, $\beta\in M_{\tau(j),n }$, $x_t\in M_j(V_t)$ such that
\begin{align}
\begin{split}
u & =\alpha \ot_{\lambda_j}(x_1,x_2,\cdots,x_m)\beta \qquad \text{with}
\\ 
\|u\|_{\lambda,n} & \leq \|\alpha\|\|x_1\|\|x_2\|\cdots \|x_m\|\|\beta\|\leq \|u\|_{\lambda,n}+\epsilon. \notag
\end{split}
\end{align}
 As $u$ is self adjoint by \Cref{O1-3}(ii), for any $\mu> 0$, we have
 \begin{align}
 \begin{split}
u & =\frac{1}{2}(u+u^*) \\
& =\frac{1}{2}(\mu \alpha\ot_{\lambda_j}(x_1,x_2,\cdots,x_m) \mu^{-1} \beta + \mu^{-1} \beta^*\ot_{\lambda_j}(x_1^*,x_2^*,\cdots,x_m^*)\mu \alpha^* ) \\
& =\left(
        \begin{array}{cc}
        \frac{\mu^{-1} \beta^*}{\sqrt{2}} & \frac{\mu\alpha}{\sqrt{2}} \\
        \end{array}
   \right)
 \left(
        \begin{array}{cc}
        0& \ot_{\lambda_j}(x_1^*,x_2^*,\cdots,x_m^*) \\
        \ot_{\lambda_j}(x_1,x_2,\cdots,x_m) & 0\\
        \end{array}
 \right)
 \left(
       \begin{array}{c}
       \frac{\mu^{-1} \beta}{\sqrt{2}} \\
        \frac{\mu \alpha^*}{\sqrt{2}}\\
        \end{array}
 \right)\\
& = \alpha P\big(\ot_{\lambda_{2j}}(v^1,v^2,\cdots,v^m)\big)P^* \alpha^* \notag
\end{split}
\end{align} where 
  $v^t=\left(
            \begin{array}{cc}
             0 & x_t^* \\
             x_t & 0 \\
            \end{array}
       \right)$,  and 
  $\alpha= \left(
                         \begin{array}{cc}
                          \frac{\mu^{-1} \beta^*}{\sqrt{2}} & \frac{\mu \alpha}{\sqrt{2}} \\
                  \end{array}
                \right)$, $R=\{1,2,\cdots,r\}\times \{r+1,\cdots,2r\}$ and  $S=\{r+1,r+2,\cdots, 2r\}\times \{1,\cdots,r\}$.\\ Thus, $\|u\|_{\lambda, n}\leq \frac{1}{2}(\mu^2\|\beta\|^2+\mu^{-2}\|\alpha\|^2)]\|v^1\| \|v^2\| \cdots \|v^m\|$, where for each $t$, $v^t$  is a self adjoint element.
                
Since, $\ds\min_{\mu>0}\frac{1}{2}(\mu^2\|\beta\|^2+\mu^{-2}\|\alpha\|^2)= \|\beta\|\|\alpha\|$, given $\delta>0$ we can choose $\mu_0>0$ such that $\|\beta\|\|\alpha\|+\delta > \frac{1}{2}(\mu_0^2\|\beta\|^2+\mu_0^{-2}\|\alpha\|^2)$. Therefore,  $\|u\|_{\lambda, n}\leq \|\tilde{\alpha}\|^2\|v^1\| \|v^2\| \cdots \|v^m\|\leq \|\beta\|\|\alpha\|\|x_1\| \|x_2\| \cdots \|x_m\|$. Thus, we get the desired norm condition. \end{proof} 
\noindent We are now in a position to define an appropriate cone structure :

\begin{definition} Let $V_1,\cdots,V_m$ be matrix ordered operator spaces and $\lambda=(\lambda_n)_{n\in \N}$ be a sequence satisfying \emph{(O1)-(O2)}. We define
\begin{enumerate}[(a)]
\item $\mathcal{C}_n$=$\{ \alpha \ot_{\lambda_j}(v_1,v_2,\cdots,v_m)\alpha^*: v_t\in M_j(V_t)^{+}, \alpha\in M_{n,\tau(j)},j\in \mathbb{N}, t=1,2,\cdots,m \} \subset M_n(\ltp V_i)_{sa}$ and
\item $M_n(\ot_{\lambda} V_i)^+ : = \mC_n^{-\|\cdot\|_{\lambda,n}}.$ 
\end{enumerate}
\end{definition}

\begin{proposition}
For matrix ordered operator spaces $V_i,$ $1 \leq i \leq m$, $\big(\ot_{i=1}^m V_i,\{\|\cdot\|_{\lambda,n}\}_{n=1}^{\infty},M_n(\ot_{\lambda} V_i)^+\big)$ is a matrix
ordered operator space.
\end{proposition}
\begin{proof}
From \Cref{selfadjoint}, the involution is an isometry on $(M_n(\ot_{\lambda} V_i))_{sa}$ hence $M_n(\ot_{\lambda} V_i)^+$ is a cone provided $\mathcal{C}_n$ is.

Let $u_1=\alpha_1\ot_{\lambda_k}
(v_1,v_2,\cdots,v_m) \alpha^{*}_1\in \mC_{n_1}, u_2=\alpha_2 \ot_{\lambda_l}(w_1,w_2,\cdots,w_m)\alpha^{*}_2\in \mC_{n},$ then using \Cref{O1-3}(ii),
\begin{align}
\begin{split}
u_1+u_2 & =(\begin{array}{cc}
            \alpha_1 &\alpha_2 \\
            \end{array})
            \left(
                   \begin{array}{cc}
                   \ot_{\lambda_k}(v_1,v_2,\cdots,v_m) & 0 \\
                    0 & \ot_{\lambda_l}(w_1,w_2,\cdots,w_m) \\
                    \end{array}
             \right)
           \left(
                  \begin{array}{c}
                   \alpha^{*}_1 \\
                   \alpha^{*}_2 \\
                   \end{array}
            \right) \\
   & =\alpha P(\ot_{\lambda_{k+l}}(x_1,x_2,\cdots,x_m)) P^* \alpha^* \in \mC_n, \notag
   \end{split}
 \end{align}
where $x_t=diag(v_t,w_t)\in M_{k+l}(V)^+$ and $\alpha=\begin{pmatrix}\alpha_1 &\alpha_2 \\
    \end{pmatrix}$, hence the family $\{C_n\}$ is closed under addition. Now, for $t\geq 0$
    \begin{align}
    \begin{split}
    tu_1 = t^{1/2}\alpha_1\ot_{\lambda_k}(v_1,v_2,\cdots,v_m) t^{1/2}\alpha^{*}_1\in \mC_{n}. \notag
    \end{split}
    \end{align}
Also, for $\gamma\in M_{m,n}$ and $\alpha  \ot_{\lambda_k}(v_1,v_2,\cdots,v_m) \alpha^*\in \mC_m$,
 \begin{align}
       \gamma^* \alpha \ot_{\lambda_k}(v_1,v_2,\cdots,v_m) \alpha^{*} \gamma  = \gamma^* \alpha  \ot_{\lambda_k}(v_1,v_2,\cdots,v_m) (\gamma^*\alpha)^* \in \mC_n. \notag
\end{align} \end{proof}

\begin{remarks} Note that 
 $\begin{pmatrix}
u  & z\\
z^*  &   u'
\end{pmatrix}\in \mC_{2n}$, implies that $$u =\begin{pmatrix}
1_n & 0 
\end{pmatrix} \begin{pmatrix}
u  & z\\
z^*  &   u'
\end{pmatrix} \begin{pmatrix}
1_n \\ 0 
\end{pmatrix} \in  \mC_n \qquad \text{and} \qquad u' =\begin{pmatrix}
0 & 1_n  
\end{pmatrix} \begin{pmatrix}
u  & z\\
z^*  &   u'
\end{pmatrix} \begin{pmatrix}
0 \\1_n 
\end{pmatrix} \in  \mC_n,$$ so using \Cref{selfadjoint} 
\begin{align*}
\begin{split}
\|u\|_{\lambda,n} & =\inf\big\{\|\alpha\|^2\|x_1\|\|x_2\|\cdots \|x_m\| :u=\alpha\ot_{\lambda_j}(x_1,x_2,\cdots,x_m)\alpha^*,\\ & \qquad \qquad \qquad \qquad \qquad \qquad \qquad \qquad \alpha\in M_{n,\tau(j)}, x_t\in M_j(V_t)^+, j\in \mathbb{N}\big\} \\
\text{and} & \\
\|u'\|_{\lambda,n} & =\inf\big\{\|\alpha'\|^2\|x_1'\|\|x_2'\|\cdots \|x_m'\| :u'=\alpha'\ot_{\lambda_r}(x_1',x_2',\cdots,x_m')\alpha'^*,\\ & \qquad \qquad \qquad \qquad \qquad \qquad \qquad \qquad \alpha'\in M_{n,\tau(r)}, x'_t\in M_r(V_t)^+, j\in \mathbb{N}\big\}. 
\notag
\end{split}
\end{align*}
\end{remarks}
Motivated by this and Han's \cite[Definition 3.2]{hanpredual}, we relate a suitable norm to the cone $\mC_n$ defined above that behaves well with matrix regular operator spaces.
\begin{definition} Let $V_i,$ $1 \leq i \leq m$ be matrix ordered operator spaces and $\lambda=(\lambda_n)_{n\in \N}$ be a sequence satisfying \emph{(O1)-(O3)}. Then for $z$ in $M_n(\ot_{\lambda}V_i)$, we define: $$\|z\|_{\Lambda,n}:=\inf \{\max\{\|u\|_{\lambda,n}, \|u'\|_{\lambda,n}\}: 
\left(\begin{matrix}
u  & z\\
z^*  &   u'
\end{matrix}\right)\in \mC_{2n} \}. $$
\end{definition}
Note that if $u \in \mC_n$, then 
 $$ \begin{pmatrix}
 u & u \\
 u & u
 \end{pmatrix} = \begin{pmatrix}
 1 \\
 1
 \end{pmatrix} u \begin{pmatrix}
 1 & 1
 \end{pmatrix} \in \mC_{2n}.$$ Therefore $\|\cdot\|_{\Lambda,n} \leq \|\cdot\|_{\lambda,n}$ on $\mC_n$. The set $\bigg\{\max\{\|u\|_{\lambda,n}, \|u'\|_{\lambda,n}\}: 
\left(\begin{matrix}
u  & z\\
z^*  &   u'
\end{matrix}\right)\in \mC_{2n} \bigg \}$ is non empty from \Cref{propercone}(i) and (ii) proved below.

\begin{proposition}\label{propercone} If $V_i,$ $1 \leq i \leq m$ are matrix regular operator spaces and $\lambda=(\lambda_n)_{n\in \N}$ a sequence satisfying \emph{(O1)-(O3)}, then
\begin{enumerate}[(i)]
\item  $\mC_n$ is a proper cone in $M_n(\ot_{i=1}^{m} V_i)$ for all $n \in \N$.
\item  For $z\in M_n(\ot_{i=1}^{m} V_i)$, there exist elements $u_1$, $u_2$ in $\mC_n$ such that
 $\left(\begin{matrix}
u_1  & z\\
z^*  &   u_2
\end{matrix}\right)\in \mC_{2n}$.
\item $\|\cdot\|_{\Lambda,n}$ is a norm on $ M_n(\ot_{i=1}^{m} V_i).$
\end{enumerate}
\end{proposition}
\begin{proof} \begin{enumerate}[(i)]
\item Let $z \in \mC_n \cap -\mC_n$, then $z= \alpha \ot_{\lambda_j}(v^1,v^2,\cdots,v^m)\alpha^* \; \text{with}\; v^t\in M_j(V_t)^+,\alpha \in M_{n,\tau(j)}.$
By \Cref{O1-3}(iii), for continuous c.p. maps $\phi^t : V_t \ra M_{k_t}$, $t=1,\ldots,m$, we have
$$(\ot_{t=1}^m \phi^t)_n(z) \in M_{nk_1k_2\ldots k_m}^+ \cap - M_{nk_1k_2\ldots k_m}^+=\{0\}.$$ Now each $V_t$ being matrix regular its dual $V_t^*$ is also matrix regular (\cite[Corollary 6.7]{schreiner}), hence each completely bounded linear map from $V_t$ into a matrix algebra is actually a linear combination of some c.p. maps. Thus, even for c.b. maps $f^{(t)} : V_t \ra M_{k_t}$, $t=1,\ldots,m$; $$(\ot_{t=1}^m f^{(t)})_n(z) \in M_{nk_1k_2\ldots k_m}^+ \cap -M_{nk_1k_2\ldots k_m}^+=\{0\},$$ which further implies the operator space injective tensor norm is given by, $$\|z\|_{ M_n(\check{\ot} V_i)}=0, \;\; \text{giving}\;\; z=0.$$
\item Using \cite[Proposition 4.1]{defantpolynomials}, for any $z \in M_n(\ot_{i=1}^{m} V_i)$ can be written as: $$z=\alpha \ot_{\lambda_j}(v_1,v_2,\cdots,v_m)\beta^* \text{ for } v_t\in M_j(V_t),\; \alpha,\beta\in M_{n,\tau(j)},\; t=1,2,\cdots,m.$$ 
As each $V_i$ is  a matrix regular operator space, there exist $v_t^1,v_t^2\in M_j(V_t)^{+}$ such that
\[x_t= \left(\begin{matrix}
v_t^1  & v_t\\
v_t^*  &   v_t^2
\end{matrix}\right)\in M_{2j} (V_t)^{+}, t=1,2,\cdots,m.\]
Using \Cref{O1-3}(ii), we have
\begin{align}
\begin{split}
& \left(\begin{matrix}
\alpha \ot_{\lambda_j}(v_1^1,v_2^1,\cdots,v_m^1)\alpha^* & \alpha \ot_{\lambda_j}(v_1,v_2,\cdots,v_m)\beta^*\\
\beta \ot_{\lambda_j}(v_1^*,v_2^*,\cdots,v_m^*)\alpha^* & \beta \ot_{\lambda_j}(v_1^2,v_2^2,\cdots,v_m^2)\beta^*   
\end{matrix}\right)\\
& = \left(\begin{matrix}
\alpha  & 0\\
0  &   \beta
\end{matrix}\right) \left(\begin{matrix}
\ot_{\lambda_j}(v_1^1,v_2^1,\cdots,v_m^1)  &  \ot_{\lambda_j}(v_1,v_2,\cdots,v_m)\\
 \ot_{\lambda_j}(v_1^*,v_2^*,\cdots,v_m^*)  &    \ot_{\lambda_j}(v_1^2,v_2^2,\cdots,v_m^2)
\end{matrix}\right)   \left(\begin{matrix}
\alpha^*  & 0\\
0  &   \beta^*
\end{matrix}\right)\\
& = \left(\begin{matrix}
\alpha  & 0\\
0  &   \beta
\end{matrix}\right) 
\ot_{\lambda_{2j}}\left( \left(\begin{matrix}
v_1^1  & v_1\\
v_1^*  &   v_1^2
\end{matrix}\right) , \left(\begin{matrix}
v_2^1  & v_2\\
v_2^*  &   v_2^2
\end{matrix}\right) ,\cdots, \left(\begin{matrix}
v_m^1  & v_m\\
v_m^*  &   v_m^2
\end{matrix}\right) \right)
   \left(\begin{matrix}
\alpha^*  & 0\\
0  &   \beta^* \notag
\end{matrix}\right) \in \mC_{2n}. 
\end{split}
\end{align} 
\item $\|\cdot\|_{\Lambda,n}$ clearly satisfies  homogeneity.

Let $z_1,z_2 \in M_n(\ot_{i=1}^m V_i)$ be any elements, then by (ii) above there exist $u_i, u_i' \in \mC_n$ such that
\begin{align*}
\begin{pmatrix}
u_i & z_i \\
z_i^* & u_i'
\end{pmatrix} \in \mC_{2n} \qquad \qquad \text{and} \qquad \qquad \|u_i\|_{\lambda,n},\|u_i'\|_{\lambda,n} < \|z\|_{\Lambda,n} + \epsilon.
\end{align*}
From $$\begin{pmatrix}
 u_1 + u_2 & z_1 + z_2 \\
 z_1^* + z_2^* & u_1' + u_2'
\end{pmatrix} \in \mC_{2n}$$
it follows that 
\begin{align}
\begin{split}
\|z_1 +z_2 \|_{\Lambda,n} & \leq  \max \{\|u_1+u_2\|_{\lambda,n}, \|u_1'+u_2'\|_{\lambda,n}\}\\
 & \leq  \max\{ \|u_1\|_{\lambda,n}+\|u_2\|_{\lambda,n}, \|u_1'\|_{\lambda,n}+\|u_2'\|_{\lambda,n}\} \\
 & <  \|z_1\|_{\Lambda,n} + \|z_2\|_{\Lambda,n} + 2 \epsilon \notag
\end{split}
\end{align}
Let $\|z\|_{\Lambda,n}=0$. Given $\epsilon >0$, there exist $u, u' \in \mC_n$ such that
\begin{align*}
\begin{pmatrix}
u & z \\
z^* & u'
\end{pmatrix} \in \mC_{2n} \qquad \qquad \text{and} \qquad \qquad \|u\|_{\lambda,n},\|u'\|_{\lambda,n} <  \epsilon.
\end{align*}
Again by \Cref{O1-3}(iii), for c.c.p. maps $f^{(t)} : V_t \ra M_{k_t}$, $t=1,\ldots,m$; \begin{align}
\begin{split}
&\begin{pmatrix}
(\ot_{t=1}^m f^{(t)})_n(u) & (\ot_{t=1}^m f^{(t)})_n(z) \\
(\ot_{t=1}^m f^{(t)})_n(z^*) & (\ot_{t=1}^m f^{(t)})_n(u')
\end{pmatrix}\\
& =(\ot_{t=1}^m f^{(t)})_{2n}\Bigg(\begin{pmatrix}
u & z \\
z^* & u'
\end{pmatrix} \Bigg)
 \in M_{2nk_1k_2\ldots k_m}^+.\notag
\end{split}
\end{align}
 It follows that \begin{align}
 \begin{split}
  \|(\ot_{t=1}^m f^{(t)})_n(z)\| & \leq \max \{\|\ot_{t=1}^m f^{(t)})_n(u) \|, \|\ot_{t=1}^m f^{(t)})_n(u') \|\}\\ & \leq \max \{ \|u\|_{\lambda,n},\|u'\|_{\lambda,n}\} < \epsilon.\notag
  \end{split}
  \end{align}
  Thus $(\ot_{t=1}^m f^{(t)})_n(z)=0$ and using matrix regularity as in case (i), $\|z\|_{ M_n(\check{\ot} V_i)}=0$, which implies $z=0$.
\end{enumerate} \end{proof}

 The positive cones of matrix ordered operator spaces are closed. Therefore, we consider $M_n(\ot_{\Lambda} V_i)^+ : =\mC_n^{-\|\cdot\|_{\Lambda,n}}$. From the definition of $\|\cdot\|_{\Lambda,n}$, since\\ $\begin{pmatrix}
u & z \\
z^* & u'
\end{pmatrix} \in \mC_{2n}$ if and only if $\begin{pmatrix}
u' & z^* \\
z & u
\end{pmatrix} =\begin{pmatrix}
 0 & 1 \\
 1 & 0
\end{pmatrix}\begin{pmatrix}
u & z \\
z^* & u'
\end{pmatrix}\begin{pmatrix}
 0 & 1 \\
 1 & 0
\end{pmatrix} \in \mC_{2n}$, we have $\|z^*\|_{\Lambda,n}= \|z\|_{\Lambda,n}$. Thus, the involution is an isometry on $(M_n(\ot_{i=1}^{m} V_i), \|\cdot\|_{\Lambda,n})$. Hence $M_n(\ot_{\Lambda} V_i)^+$ is a proper cone.

Recall from \cite[Definition 4.7]{wiesnerpolynomials} that an operator space matrix norm $\|\cdot\|_{\alpha}$ is said to be $\lambda$-subcross if
$$\|\ot_{p}(v_1,v_2,\ldots,v_m)\|_{(\alpha,\tau(p))} \leq \|v_1\|\ldots\|v_m\|$$
for all $p \in \N$ and $v_i \in M_p(V_i).$ In case of equality the norm is called $\lambda$-cross.

\begin{theorem}
For matrix regular operator spaces $V_i$ and $\lambda=(\lambda_n)_{n\in \N}$ satisfying \emph{(O1)-(O3)}, $(\ot_{i=1}^m V_i,\{\|\cdot\|_{\Lambda,n}\}_{n=1}^{\infty}, M_n(\ot_{\Lambda} V_i)^{+})$ is a  matrix regular operator space with $\lambda$-subcross matrix norm.
\end{theorem}
\begin{proof}
We first prove that $M_n(\ot_{i=1}^m V_i)$ is an operator space with the family $\{\|\cdot\|_{\Lambda,n}\}_{n=1}^{\infty}$ of matrix norms.\\ Given  $z_1\in M_{n_1}(\ot_{i=1}^m V_i)$,  $z_2\in M_{n_2}(\ot_{i=1}^m  V_i)$ and $\epsilon>0$, choose $u_i,u_i' \in \mC_n$ such that $\left (\begin{matrix}
u_i  & z_i\\
z_i^* &   u_i'
\end{matrix}\right) \in \mC_{2n_i}$   and $\|u_i\|_{\lambda,n_i}, \|u_i'\|_{\lambda,n_i} < \|z_i\|_{\Lambda,n_i} + \epsilon$, $i=1,2$.\\ By definition, there exist representations  $$u_i=\alpha_i \ot_{\lambda_{j_i}} (v_1^i,v_2^i,\cdots,v_m^i) \alpha_i^* \qquad \qquad \alpha_i\in M_{n_i, \tau(j_i)}, v_t^i\in M_{j_i}(V_t)^+.$$ 
Since $\|\cdot\|_{\lambda,n}$ is an operator space norm, using Ruan's first condition (M1) \cite{erbook} for operator space, we have
\begin{align}
\begin{split}
\|u_1\oplus u_2\|_{\lambda,2n} 
& \leq  
  \max\{\|u_1\|_{\lambda,n}, \|u_2\|_{\lambda,n}\} \notag
\end{split}
\end{align} 
As,
\begin{align}
\begin{split}
& \begin{pmatrix}
u_{1} & 0  & z_1 & 0\\
0 & u_2  & 0 & z_2\\
z_{1}^* & 0  & u_1^{'} & 0\\
0& z_2^* & 0 & u_2^{'}
\end{pmatrix}\\
& = \begin{pmatrix}
1 & 0  & 0 & 0\\
0 & 0 & 1 & 0\\
0 & 1  & 0& 0\\
0& 0 & 0 & 1\end{pmatrix}
\begin{pmatrix}
u_1 & z_1  & 0  & 0 \\
z_1^* &   u_1' & 0 & 0\\
0& 0  &u_2  & z_2\\
0& 0 &  z_2^* &   u_2'
\end{pmatrix}
\begin{pmatrix}
1 & 0  & 0 & 0\\
0 & 0 & 1 & 0\\
0 & 1  & 0& 0\\
0& 0 & 0 & 1
\end{pmatrix} \in \mC_{2(n_1+n_2)} \notag
\end{split}
\end{align} we have,
\begin{align}
\begin{split}
\bigg\|\begin{pmatrix}
z_1 & 0  \\
0 & z_2 
 \end{pmatrix}\bigg\| _{\Lambda,2n} 
& \leq \max \bigg\{\bigg\|\begin{pmatrix}
u_1 & 0  \\
0 & u_2  \end{pmatrix}\bigg\|_{\lambda,2n},\bigg\|
\begin{pmatrix}
u_1' & 0  \\
0 & u_2'  \end{pmatrix}\bigg\|_{\lambda,2n}\bigg\}\\
& \leq \max \{ \|u_1\|_{\lambda,n},\|u_2\|_{\lambda,n},\|u_3\|_{\lambda,n},\|u_4\|_{\lambda,n}\}\\
& < \max \{\|z_1\|_{\Lambda,n},\|z_2\|_{\Lambda,n}\} + \epsilon. \notag
\end{split}
\end{align}
Let $z\in M_n(\ot_{i=1}^m V_i)$ and $\alpha, \beta \in M_{m,n}$, then there exist $u, u' \in \mC_n$ such that
 $\begin{pmatrix}
u  & z\\
z* &   u'
\end{pmatrix} \in \mC_{2n}$   and $\|u\|_{\lambda,n}, \|u'\|_{\lambda,n} < \|z\|_{\Lambda,n} + \epsilon.$ Assuming, $\|\alpha\|=\|\beta\| $ by homogeneity, since $$\begin{pmatrix}
\alpha u \alpha^* & \alpha z \beta \\
\beta^* z^* \alpha^* & \beta^* u' \beta  
\end{pmatrix} = \begin{pmatrix}
\alpha & 0 \\
0 & \beta^*
\end{pmatrix} \begin{pmatrix}
u & z \\
z^* & u'
\end{pmatrix} \begin{pmatrix}
\alpha^* & 0 \\
0 & \beta 
\end{pmatrix} \in \mC_{2n}$$
we have,
\begin{align}
\begin{split}
\|\alpha z \beta \|_{\Lambda,n} & \leq \max \{ \|\alpha u \alpha^*\|_{\lambda,n}, \|\beta^* u' \beta \|_{\Lambda,n}\}\\
& \leq \max \{ \|\alpha\|^2\|u\|_{\lambda,n}, \|\beta\|^2\|u'\|_{\lambda,n}\} \\
 & < \|\alpha \| \|\beta\|(\|z\|_{\Lambda,n}+\epsilon). \notag
 \end{split}
 \end{align} 
 Hence, $\big(\ot_{i=1}^m  V_i, \{\|\cdot\|_{\Lambda,n}\}_{n=1}^\infty\big)$ is an operator space.\\
 Let $v_i \in M_j(V_i)$ with $\|v_i\| < 1$ for $i=1,\ldots,m$. Then there exist $u_i,u_i' \in M_j(V_i)^+_{\|\cdot\| \leq 1}$ such that $$\begin{pmatrix}
 u_i & v_i \\
 v_i^* & u_i'
\end{pmatrix} \in M_{2j}(V_i)^+ \qquad \qquad i=1,2,\ldots,m.$$
Now, 
\begin{align}
\begin{split}
 & \begin{pmatrix}
 \ot_{\lambda_j} (u_1,\ldots,u_m)  & \ot_{\lambda_j} (v_1,\ldots,v_m)\\
 \ot_{\lambda_j} (v_1,\ldots,v_m)^* & \ot_{\lambda_j} (u_1',\ldots,u_m')
 \end{pmatrix}\\
 & = P \ot_{\lambda_{2j}} \big( \begin{pmatrix}
  u_1 & v_1 \\
 v_1^* & u_1'
 \end{pmatrix}, \ldots, 
 \begin{pmatrix}
  u_m & v_m \\
 v_m^* & u_m'
 \end{pmatrix}\big) P^* \in  \mC_{2\tau(j)}, \notag
 \end{split}
 \end{align}
 it follows that \begin{align}
 \begin{split}
 \|\ot_{\lambda_j} (v_1,\ldots,v_m)\|_{\Lambda,2j} & \leq \max \{\|\ot_{\lambda_j} (u_1,\ldots,u_m) \|_{\lambda,j},\|\ot_{\lambda_j} (u_1',\ldots,u_m')\|_{\lambda,j} \} \\ & \leq \max \{\|u_1\|\ldots \|u_m\|,\|u_1'\|\ldots \|u_m'\|\} \\
 & <1 \notag
 \end{split}
 \end{align}
 Therefore, the family of matrix norms $\{ \|\cdot\|_{\Lambda,n}\}_{n=1}^{\infty}$ is $\lambda$-subcross.\\
 As $\|u\|_{\Lambda,n} \leq \|u\|_{\lambda,n},$ if $\|z\|_{\Lambda,n} < 1$ then there exist  $u ,u' \in \mC_n$ such that $$\begin{pmatrix}
 u & z \\
 z^* & u'
 \end{pmatrix} \in \mC_{2n} \qquad \text{and} \qquad \|u\|_{\lambda,n},\|u'\|_{\lambda,n} < 1.$$
 Thus, $$\|u\|_{\Lambda,n} < 1 \qquad  \text{and} \qquad  \|u'\|_{\Lambda,n}  < 1,$$ and matrix regularity follows. \end{proof}

\section{$\lambda$-operator system tensor product}\label{s:3}

We now prove that the cones $\mC_n$ associated with $\lambda$ under the conditions (O1)-(O3) also preserve the operator system structure defined in \cite{KPTT1}. The techniques are again same as that for the $\max$ operator system tensor product defined in \cite{KPTT1}.
\begin{theorem}
Let $(\mS,\{M_n(\mS)^+\}_{n=1}^{\infty},1_{\mS})$ and $(\mT,\{M_n(\mT)^+\}_{n=1}^{\infty},1_{\mT})$ be operator systems. The family $\{\mC_n\}_{n=1}^{\infty}$ associated with a sequence $\lambda=(\lambda_n)_{n\in \N}$ satisfying \emph{(O1)-(O3)}, is a matrix ordering on $\mS \ot \mT$ with order unit $1_{\mS} \ot 1_{\mT} $.
\end{theorem}
 \begin{proof}
 From \Cref{propercone}, we know that $\{\mC_n\}_{n=1}^\infty$ is a family of proper compatible cones on $M_n(\ot_{\lambda} \mS_i)$. We only need to check that $1  \ot 1 $ is a matrix order unit.
\\
Let $\alpha \ot_{\lambda_j} (s,t)\alpha^* \in (\mS_1 \ot \mS_2)_{sa}$ with $s \in \big(M_j(\mS)\big)_{sa}$ , $t \in \big(M_j(\mT)\big)_{sa}$ and $\alpha \in M_{1,\tau(j)},$ $1_{\mS}$  and $1_{\mT}$ being Archimedean order unit for $\mS$ and $\mT$ respectively, then we can find $K$ large enough such that $$K (1_{\mS})_j + s \in M_j(\mS)^+ \qquad \text{ and} \qquad K (1_{\mS})_j - s \in M_j(\mS)^+, $$ 
$$K (1_{\mT})_j + t \in M_j(\mT)^+ \qquad \text{ and} \qquad K (1_{\mT})_j - t \in M_j(\mT)^+. $$
So that, $$\ot_{\lambda_j}\big(K (1_{\mS})_j + s ,K (1_{\mT})_j + t\big),  \ot_{\lambda_j}\big(K (1_{\mS})_j - s , K (1_{\mT})_j - t\big) \in \mC_{n}.$$ 
Further,
 \begin{align}
\begin{split}
\mC_{n} \ni&  \alpha \ot_{\lambda_j}\big(K (1_{\mS})_j + s ,K (1_{\mT})_j + t\big)\alpha^* + \alpha \ot_{\lambda_j}\big(K (1_{\mS})_j - s ,K (1_{\mT})_j - t\big) \alpha^*\\
 & = \alpha \Big( \ot_{\lambda_j}\big( K (1_{\mS})_j,K (1_{\mT})_j \big) + \ot_{\lambda_j}\big(s,t \big) \Big)\alpha^*\\
 & =  \alpha \Big( \Big( K^2  \lambda_j( 1 ,1 ) \ot 1_{\mS} \ot 1_{\mT} \big) + \ot_{\lambda_j}\big(s,t \big) \Big)\alpha^*\\
 &= \alpha \Big( \big( K^2 I_{\tau(j)} \ot 1_{\mS} \ot 1_{\mT} \Big) + \ot_{\lambda_j}\big(s,t \big) \Big)\alpha^* \\
 & = \alpha  \Big( K^2 (1_{\mS} \ot 1_{\mT})_{\tau(j)}  + \ot_{\lambda_j}\big(s,t \big) \Big)\alpha^*\\
 & =  K^2  \alpha(1_{\mS} \ot 1_{\mT})_{\tau(j)} \alpha^*+ \alpha\Big(\ot_{\lambda_j}\big(s,t \big) \Big)\alpha^*\\
 & = (K^2 \alpha \alpha^*) 1_{\mS} \ot 1_{\mT} + \alpha\Big(\ot_{\lambda_j}\big(s,t \big) \Big)\alpha^*  \notag
 \end{split}
 \end{align}
which proves that $ 1_{\mS} \ot 1_{\mT}$ is an order unit. Similarly, one can prove that $1_{\mS} \ot 1_{\mT}$ is in fact a matrix order unit. \end{proof}
 \begin{definition} Let $\lambda=(\lambda_n)_{n \in \N}$ fulfills conditions \emph{(O1)-(O3)}, and $$\mC_n^{\lambda}:=\{ P \in M_n(\mS\ot \mT) : r (1_{\mS} \ot 1_{\mT})_n + P \in \mathcal{\mC}_n, \ \forall
r > 0 \}
$$ be the Archimedeanization(\cite{PT}) of the matrix ordering $\mC_n$ for all $n \geq 1$.
We call the operator system $(\mS \otimes \mT, \{\mC_n^{\lambda}\}_{n=1}^{\infty}, 1_\mS \ot 1_\mT)$ the {\it $\lambda$-  operator system tensor product} of $\mS$\ and $\mT$ and denote it by $\mS\ot_{\lambda}\mT$.
\end{definition}

\begin{theorem}
The mapping $\lambda : \mO \times \mO \ra \mO$ sending $(\mS,\mT)$ to $\mS \ot_{\lambda} \mT$ is an operator system tensor product in the sense of \cite{KPTT1}.
\end{theorem}
\begin{proof}
Observe that,
\begin{enumerate}
\item[(T1)] By definition $(\mS \ot \mT,\{\mC_n^{\lambda}(\mS \ot \mT)\}_{n=1}^{\infty},1_{\mS} \ot 1_{\mT})$ is an operator system.
\item[(T2)] For $P \in M_k(\mS)^+$ and $Q \in M_l(\mT)^+$, since $$P \ot Q=\alpha \ot_{\lambda_{k+l}}\big(I_{k+l} \ot P,I_{k+l} \ot Q\big) \alpha^*\in \mC_{kl},$$ where $\alpha=(I_{k+l},0,\cdots,0) \in M_{kl,\tau(k+l)}$, we have property (T2).
\item[(T3)] For unital completely positive maps $\phi \in \mS\ra M_n$ and $\psi \in \mT \ra M_m$, using \Cref{O1-3}(iii) we have $(\phi \ot \psi)_{n}(\mC_{n}) \subseteq  M_n^+$, thus (T3) follows.
\item[(T4)] Let $\phi \in UCP(\mS_1,\mS_2)$ and $\psi \in UCP(\mT_1,\mT_2)$. Then for any element $A\ot_{\lambda_j}(P, Q)A^* \in \mC_n$, where $A \in M_{n,\tau(j)}$, $P \in M_j(\mS_1)^+$ and $Q \in M_j(\mT_1)^+$:
\begin{align}
\begin{split}
(\phi \ot \psi)_n(A\ot_{\lambda_j}(P, Q)A^*)& = A \big((\phi \ot \psi)_{\tau(j)}\ot_{\lambda_j}(P, Q)\big)A^* \\
&= A \ot_{\lambda_j} \big( \phi_j (P),\psi_j(Q)\big)A^* \in M_n(\mS_2 \ot_{\lambda} \mT_2)^+.\notag
\end{split}
\end{align}Thus, $(\phi \ot \psi)_n(\mC_n(\mS_1,\mT_1)) \subseteq \mC_n(\mS_2,\mT_2),$ and using \cite[Lemma 2.5]{KPTT1} $\phi \ot \psi \in UCP(\mS_2,\mT_2).$\end{enumerate}
 \end{proof}
\begin{remarks} If $\lambda$ is either Kronecker or Schur Product,  the cone $\mC_n^{\lambda}$ coincides with $\mC_n^{\max}=\mC_n^{s}$ (\cite{KPTT1, ng}).
\end{remarks}
\section{$\lambda$-tensor product of $C^*$-algebras}\label{s:4}
 
 We now move on to the algebraic structures for the $\lambda$-theory. For this we make use of the condition (W2) (\Cref{w2}).
 
An associative algebra $A$ over $\C$ is said to be a completely contractive Banach algebra  if it is a complete operator space for which the multiplication map $m_A: A\times A \to A$ $(a,b)\to ab$ is jointly completely contractive, i.e. $\|[a_{ij}b_{kl}]\|\leq \|[a_{ij}]\| \|[b_{kl}]\|$ for all $[a_{ij}]\in M_n(A)$ and $[b_{kl}]\in M_n(A)$. 

\begin{theorem}
For completely contractive Banach algebras $A_1, A_2,\cdots, A_m$, $\ot^\lambda A_i$ is a Banach algebra if $\lambda$ satisfies \emph{(W2)}. Further if each $A_i$ is a Banach $*$-algebra and $\lambda$ also satisfies \emph{(O1)}, then $\ot^\lambda A_i$ is a Banach $^*$-algebra. Moreover if each $A_i$  is approximately unital then $\ot^\lambda A_i$ is also approximately unital.
\end{theorem}
\begin{proof}
Let $x$, $y\in \ot_\lambda A_i$ with $$x=\alpha \ot_{\lambda_r}(u^{(1)},u^{(2)},\ldots,u^{(m)})\beta \:\:\text{and}\;\; y=\gamma \ot_{\lambda_s}(v^{(1)},v^{(2)},\ldots,v^{(m)})\delta,$$ where for each $t=1,2,\cdots,m$ 
\begin{align}
\begin{split}
\alpha\in M_{1, \tau(r)}, \;\; \beta\in M_{\tau(r),1},\;\;  & \gamma \in M_{1, \tau(s)}, \;\; \delta \in M_{\tau(s),1} \\ u^{(t)}:= \ds \sum_{(i_t,j_t)} \vep^{[r]}_{i_t,j_t} \otimes u_{i_t,j_t}^{(t)},\; \;\; & v^{(t)}:= \ds\sum_{(k_t,l_t)} \vep^{[s]}_{k_t,l_t} \otimes v_{k_t,l_t}^{(t)}. \notag
\end{split}
\end{align} Then, using Property (W2), there exists $S \in M_{\tau(r)\tau(s),\tau(rs)}, T \in M_{\tau(rs),\tau(r)\tau(s)}$ with $\|S\|,\|T\| \leq 1$ such that
\begin{align}
\begin{split}
 xy & = \big( \sum_{i_m,j_m} \ldots \sum_{i_1,j_1} \alpha \lambda_r (\vep^{[r]}_{i_1,j_1},\ldots,\vep_{i_m,j_m})\beta \ot u_{i_1,j_1}^{(1)} \ot \ldots \ot u_{i_m,j_m}^{(m) } \big) \\
& \qquad \qquad \big( \sum_{k_m,l_m} \ldots \sum_{k_1,l_1} \gamma \lambda_s (\vep^{[s]}_{k_1,l_1},\ldots,\vep_{k_m,l_m})\delta \ot v_{k_1,l_1}^{(1)} \ot \ldots \ot v_{k_m,l_m}^{(m) } \big)\\
&= \sum_{i_m,j_m} \ldots \sum_{i_1,j_1}\sum_{k_m,l_m} \ldots \sum_{k_1,l_1} \big((\alpha \ot \gamma)( \lambda_r (\vep^{[r]}_{i_1,j_1},\ldots,\vep_{i_m,j_m}) \ot \lambda_s (\vep^{[s]}_{k_1,l_1},\ldots,\vep_{k_m,l_m})) \\ & \qquad \qquad \qquad\qquad\qquad\qquad (\beta\ot \delta) \ot  u_{i_1,j_1}^{(1)}v_{k_1,l_1}^{(1)} \ot \ldots \ot u_{i_m,j_m}^{(m) }v_{k_m,l_m}^{(m) } \big) \\
& \overset{\text{\textbf{(W2)}}}{=} (\alpha \ot \gamma) S ( \sum_{i_m,j_m} \ldots \sum_{i_1,j_1}\sum_{k_m,l_m} \ldots \sum_{k_1,l_1}  (\lambda_{rs}(\vep^{[r]}_{i_1,j_1} \ot \vep^{[s]}_{k_1,l_1} ,\ldots,\vep_{i_m,j_m} \ot \vep^{[s]}_{k_1,l_1} ) \\
& \qquad \qquad\qquad\qquad\qquad\qquad\ot u_{i_1,j_1}^{(1)}v_{k_1,l_1}^{(1)} \ot \ldots \ot u_{i_m,j_m}^{(m) }v_{k_m,l_m}^{(m) } )T (\beta\ot \delta)\\
& = (\alpha \ot \gamma) S ( \sum_{i_m,j_m} \ldots \sum_{i_1,j_1}\sum_{k_m,l_m} \ldots \sum_{k_1,l_1}  (\lambda_{rs}(\vep^{[rs]}_{(i_1-1)s+k_1,(j_1-1)s)+l_1} ,\ldots, \\
& \qquad \qquad\qquad\qquad \vep^{[rs]}_{(i_m-1)s+k_m,(j_m-1)s)+l_m}) \ot u_{i_1,j_1}^{(1)}v_{k_1,l_1}^{(1)} \ot \ldots \ot u_{i_m,j_m}^{(m) }v_{k_m,l_m}^{(m) } )T (\beta\ot \delta)\\
&= (\alpha \ot \gamma) S (\ot_{\lambda_{rs}}( z^{(1)},\ldots,z^{(m)}) )T (\beta\ot \delta); 
\end{split}
\end{align}

where $z^{(t)}=  \sum_{i_t,j_t} \sum_{k_t,l_t} \vep^{[rs]}_{(i_t-1)s+k_t,(j_t-1)s)+l_t} \ot u_{i_1,j_1}^{(t)}v_{k_t,l_t}^{(t)}=u^{(t)}\ot v^{(t)};\,t=1,\ldots,m$. \\Thus, \begin{align}
\begin{split}
\|xy\|_{\lambda,1} & \leq \|\alpha \ot \gamma\| \|S\|\|z^{(1)}\|\|z^{(2)}\|\cdots \|z^{(m)}\| \|T\|\|\beta\ot \delta\|\\
& \leq \|\alpha\| \| \gamma\| \|v^{(1)}\|\|w^{(1)}\|   \ldots \|v^{(m)}\|\|w^{(m)}\|\|\beta\|\| \delta\|\notag
\end{split}
\end{align} making $\ot_{\lambda}^m A_i$, and hence $\ot^{\lambda} A_i$ a Banach algebra.

If $\lambda$ satisfies (O1), $*$-part follows from \Cref{O1-3}(i) and definition of $\|\cdot\|_{\lambda,1}$.

One can easily verify that $\|\cdot \|_{\lambda, 1} \leq \|\cdot\|_{\gamma}$, implying that $\|\cdot \|_{\lambda, 1} $  is an
admissible cross norm on $\ot_{m} A_i$ . Therefore, $\ot^\lambda A_i$ has a bounded approximate
identity whenever each $A_i$ is approximately unital.\end{proof}

In particular, we have the following well known result (see\cite{kumarideal,rajpalschur}):
\begin{corollary} $\ot^\otimes A_i$, the projective tensor product and  $\ot^\odot A_i$, the Schur tensor product are  Banach $^*$-algebras with a bounded approximate identity, however, $\ot^{\bullet} A_i$, the Haagerup tensor product is a Banach algebra.
\end{corollary}

In general, $\lambda$-tensor product of operator spaces is not injective. Since, $(\ot^{\lambda} A_i)^*=CB_{\lambda}(A_1\times A_2 \cdots A_m, \C)$\cite[Proposition 4.11]{wiesnerpolynomials} completely isometrically, so the proof of 
\cite[Theorem 5]{kumarideal} can be adopted in this case to show the injectivity of $\lambda$-tensor product for the closed ideals, i.e.

\begin{proposition}
Let $I_i$  be closed two-sided ideals in $C^{*}$-algebras $A_i$ for $ i=1,2,\cdots,m$, then $\ot^{\lambda}I_i$ is a closed two-sided $^*$-ideal of $\ot^{\lambda}A_i$.
\end{proposition}

\begin{lemma}\label{sy6}
Let $W_i,$ $1 \leq i \leq m$ be completely complemented subspaces of the operator spaces $V_i,$  $1\leq i \leq m$ complemented by cb
projection having  cb norm equal to 1, respectively, then $\otimes^{\lambda} W_i$ is a closed subspace of $\otimes^{\lambda} V_i$.
\end{lemma}
\begin{proof}
Using the assumption, there are cb projections  $P_1, P_2, \cdots, P_m$ from  $V_1$ onto  $W_1$, $V_2$ onto $W_2   \cdots$, $V_m$ onto $W_m$ with $\|P_1\|_{cb}=\|P_2\|_{cb}=\cdots=\|P_m\|_{cb}=1$.  Therefore, by the functoriality of the $\lambda$- tensor product(\cite[Proposition 6.1]{wiesnerpolynomials}, $\otimes^{m}_{i=1} P_i: \otimes^{\lambda} V_i \to \otimes^{\lambda} W_i$ is a completely bounded map and $\| \otimes^{m}_{i=1}  P_i\|_{cb}\leq 1$. Since, for $u\in  \ds\otimes_{i=1}^{m} V_i$, $\otimes^{m}_{i=1} P_i(u)=u$, so $\|u\|_{\otimes^{\lambda} W_i}\leq \|u\|_{\otimes^{\lambda}  V_i}$, hence $\otimes^{\lambda}  W_i$ is a closed subspace of $ \otimes^{\lambda} V_i$.\end{proof}

Since, there is  a conditional expectation from a $C^*$-algebra $A$ onto a finite dimensional $C^*$-subalgebra of $A$, so  by the above Lemma for finite dimensional $C^*$-algebras, $\lambda$-tensor product of operator spaces is injective. In general $\|\cdot\|_{\lambda}$ need not be injective.\\

 However, for $\widehat{\otimes}$, we have something partial:
\begin{proposition}\label{on1110}
Let $A_0$ and $B_0$ be closed $^*$-subalgebras of $A$ and $B$, respectively, then $A_0 \widehat{\otimes} B_0$ is \emph{(isomorphic to)} closed $^*$-subalgebra of $A\widehat{\otimes} B$.
\end{proposition}

\begin{proof}
Let $I$ denote the closure of $A_0 \otimes B_0$ in $A\widehat{\otimes} B$, so that $I$ is a closed $^*$-subalgebra of
$A\widehat{\otimes} B$. We first claim that $\|u\|_{A\widehat{\otimes} B}\leq \|u\|_{A_0 \widehat{\otimes} B_0}
\leq 2\|u\|_{A\widehat{\otimes} B}$ for  $u\in A_0 \otimes B_0$. Choose $f\in (A_0 \widehat{\otimes} B_0)^*$ such that
$f(u)=\|u\|_{A_0 \widehat{\otimes} B_0}$ with $\|f\|=1$.  Let $\phi_0$ be the  jcb bilinear form on
$A_0 \times B_0$ corresponding to $f$.  By (\cite[Corollary 3.10]{erconj}), $\phi_0: A_0 \times B_0\to \mathbb{C} $
extends to a jcb bilinear form $\phi: A \times B\to \mathbb{C} $ such that $\|\phi\|_{jcb}\leq 2\|\phi_0\|_{jcb}$. Therefore $\|\tilde{f}\| \leq 2$,
where $\tilde{f}$ is the linear functional on $A\widehat{\ot} B$ corresponding to $\phi$, and thus the claim. Now
consider the identity map $i:(A_0 \otimes B_0,\|\cdot\|_{A_0\widehat{\ot}B_0})\to (A \otimes B,\|\cdot\|_{A\widehat{\ot}B})$ which is linear and continuous by the last claim, so it can be extended to $\widetilde{i}:A_0 \widehat{\otimes} B_0\to A \widehat{ \otimes} B$.
We now show that $A_0 \widehat{\otimes} B_0$ is isomorphic to $I$.  For the injectivity of $\widetilde{i}$, by \cite[Theorem 2]{jainideal2}, it is enough to show
that it is injective on $A_0 \otimes B_0$ but this follows directly by the last inequality. Again, by the last inequality,
$\widetilde{i}^{-1}$ is continuous.  For onto-ness, let $u\in I$. There is a sequence $u_n\in A_0 \otimes B_0$
 converging to $u$ in $\|\cdot\|_{A\widehat{\otimes} B}$-norm. The sequence $\{u_n\}$ becomes Cauchy with respect to
 $\|\cdot\|_{A_0\widehat{\otimes} B_0}$-norm by the last claim, so it converges, say, to $u'$.  Clearly,
 $\widetilde{i}(u')=u$. Thus $A_0 \widehat{\otimes} B_0$ can be regarded as a closed $^*$-subalgebra of $A\widehat{\otimes} B$. \end{proof}

\begin{proposition}\label{sc119}
For $C^*$-algebras $A$ and $B$, any  $\lambda$-cb bilinear form $\phi$ on $A\times B$  can be extended uniquely to  $\tilde{\phi}$ on $A^{**}\times B^{**}$ such that $\|\phi\|_{\lambda}=\|\tilde{\phi}\|_{\lambda}$.
 \end{proposition}
\begin{proof}
Since $\phi: A\times B\to \C$ is $\lambda$-cb bilinear form. It is in particular bounded bilinear form and thus determines a unique separately normal bilinear form  $\tilde{\phi}: A^{**} \times B^{**} \to \C$ by \cite[Corollary 2.4]{haaggroth}. For  $k\in \mathbb{N}$, consider the map $\tilde{\phi}_k: M_k(A^{**}) \times M_k(B^{**}) \to M_{\tau(k)}$ taking $\tilde{\phi}_k(a_1\ot m, a_2\ot m') = \lambda_k(a_1,a_2)\ot \tilde{\phi}(m,m') $. Let $a^{**}=[a_{ij}^{**}]\in M_k(A^{**})$ and $b^{**}=[b_{ij}^{**}]\in M_k(B^{**})$ with
 $\|a^{**}\|\leq 1$ and $\|b^{**}\|\leq 1$. Since the unit ball of  $M_k(A)$ is $w^*$-dense in the unit ball of  $M_k(A^{**})$, so we obtain  a net $(a_{\lambda})=(a_{ij}^{\lambda})$ (resp., $(b_\nu)$) in $M_k(A)$ (resp., $M_k(B)$) which
 is $w^*$-convergent to $a^{**}$ (resp., $b^{**}$) with $\|a_{\lambda}\|\leq 1$ (resp., $\|b_{\nu}\|\leq 1$). Therefore, $\widehat{a_{ij}^{\lambda}}$ is $w^*$-convergent to $a_{ij}^{**}$ for each $i,j$. Now by the separate normality  of $\tilde{\phi}$, we have
 $\|\lambda_k\ot \tilde{\phi}(\ds\sum_{i,j}\epsilon_{ij} \ot a_{ij}^{**}, \ds\sum_{p,l}\epsilon_{pl} \ot b_{pl}^{**})\|= \|\ds\sum_{i,j,p,l} \lambda_k(\epsilon_{ij}, \epsilon_{pl}) \ot \tilde{\phi}( a_{ij}^{**},b_{pl}^{**})\|=\|\ds\lim_\lambda \ds\lim_\nu \ds\sum_{i,j, p,l} \lambda_k(\epsilon_{ij}, \epsilon_{pl}) \ot \tilde{\phi}( a_{ij}^{\lambda},b_{pl}^{\nu})\|=\ds\lim_\lambda \ds\lim_\nu\| \ds\sum_{i,j, p,l} \lambda_k(\epsilon_{ij}, \epsilon_{pl}) \ot \phi( a_{ij}^{\lambda},b_{kl}^{\nu})\|\leq \|\lambda_k\ot \phi\|$ for each $k\in \mathbb{N}$. Thus $\|\tilde{\phi}_k\|\leq \|\phi_k\|\leq \|\phi\|_{\lambda}$  for every $k\in \mathbb{N}$.  Clearly, $\|\phi\|_{\lambda}\leq \|\tilde{\phi}\|_{\lambda}$ as $\phi$ being the restriction of $\tilde{\phi}$. Hence $\|\phi\|_{\lambda}=\|\tilde{\phi}\|_{\lambda}$.
\end{proof}
For a tensor norm $\alpha$ and a closed ideal $J$ of $A\ot^{\alpha} B$, we try to find out whether $a\ot b\in J_{\min}$ implies that $a\ot b\in J$. This question
stems from the study of the elusive nature of the Haagerup tensor product of $C^*$-algebras, it was resolved  for the Haagerup  tensor product in (\cite[Theorem 4.4]{allenideal}) and for the operator space projective tensor product in \cite[Theorem 6]{kumarideal}. We present here a unified approach.\\For $C^*$-algebras $A$ and $B$, assume that $\|\cdot\|_{\lambda}\geq \|\cdot\|_{\min}$ on $A \ot B$, so there will be a identity map $i$ from $A\otimes^{\lambda} B$ into $A\otimes^{\min} B$. 

\begin{lemma}\label{lemma}
Let $M$ and $N$ be von Neumann algebras and let $L$ be a closed ideal in $M\otimes^{\lambda}N$, where the sequence $\lambda=(\lambda_n)_{n\in \N}$ satisfies condition \emph{(W2)}. If $1\otimes1 \in L_{\min},$ then $1 \otimes 1 \in L$ and $L$ equals $M\otimes^{\lambda}N$.
\end{lemma} 
\begin{proof}
Since, $1\otimes1 \in L_{\min}$, for a given $\epsilon=\dfrac{1}{2}$, there exists  $w\in L$ such that  $\|i(w)-1\otimes1\|_{\min}< \dfrac{1}{2}$. Let $w= \displaystyle\sum_{t=1}^{\infty}\alpha_t \ot_{\lambda_{r_t}}(u^{(t)},v^{(t)}) \beta_t;$  be a norm convergent representation in $M\otimes^{\lambda}N$ where for $t=1,2,\ldots,\;\; r_t \in \N$,
\begin{align}
\begin{split}
\alpha_t \in M_{1, \tau(r_t)}, \;\; \beta_t \in M_{\tau(r_t),1},\;\;  u^{(t)}:= \ds \sum_{(i_t,j_t)} \vep^{[r_t]}_{i_t,j_t} \otimes u_{i_t,j_t}^{(t)},\; \;\; & v^{(t)}:= \ds\sum_{(k_t,l_t)} \vep^{[r_t]}_{k_t,l_t} \otimes v_{k_t,l_t}^{(t)}. \notag
\end{split}
\end{align}
By \cite[ Theorem 8.3.5]{krfundamentals}, there exist sequences $\{x_{i_t,j_t}^{(t)}\}\in Z(M)$, $\{y_{i_t,j_t}^{(t)}\}\in Z(N)$,
$\{\phi_{n}\}\in P(M)$ and $\{\psi_{n}\}\in P(N)$ such that
\begin{eqnarray}\label{b3}
  \displaystyle\lim_{n\rightarrow \infty} \|\phi_{n}( u_{i_t,j_t}^{(t)})-  x_{i_t,j_t}^{(t)}\|=\displaystyle\lim_{n\rightarrow \infty} \|\psi_{n}(v_{k_t,l_t}^{(t)})- y_{i_t,j_t}^{(t)}\|=0\;\;(t=1,2,\ldots)
\end{eqnarray}
where $P(M)$ denotes the set of all mappings $\phi: M \to M$ such that, for $m\in M$, $\phi(m)$ is in the convex hull of the set $\{\mathrm{umu}^*: \mathrm{u} \in U(M)\}$.\\
For each $n\in \mathbb{N}$, using the contractive maps $\phi_{n}\otimes \psi_{n}$ on $M\otimes^{\lambda} N $ (\cite[Proposition 6.1]{wiesnerpolynomials}) and invariance of ideal $L$ under $\phi_{n}\otimes \psi_{n}$, we have for all positive integers $k\leq l$

\begin{align}
\begin{split}
& \left\| \displaystyle\sum_{t=k}^{l} \displaystyle\alpha_t \big(\sum_{(k_t,l_t)} \sum_{(i_t,j_t)}  \lambda_{r_t} (\vep^{[r_t]}_{i_t,j_t} , \vep^{[r_t]}_{k_t,l_t})\otimes  x_{i_t,j_t}^{(t)} \otimes y_{k_t,l_t}^{(t)}\big) \beta_t \right\|_{\lambda,1} \\
& \leq \left\|  \displaystyle\sum_{t=k}^{l} \alpha_t \big(\sum_{(k_t,l_t)} \sum_{(i_t,j_t)}  \lambda_{r_t} (\vep^{[r_t]}_{i_t,j_t} , \vep^{[r_t]}_{k_t,l_t})\otimes  x_{i_t,j_t}^{(t)} \otimes y_{k_t,l_t}^{(t)}\big) \beta_t \right. \\ & 
\qquad \qquad- \left. \displaystyle\sum_{t=k}^{l} \alpha_t \big(\sum_{(k_t,l_t)} \sum_{(i_t,j_t)}  \lambda_{r_t} (\vep^{[r_t]}_{i_t,j_t} , \vep^{[r_t]}_{k_t,l_t}) \otimes \phi_n(u_{i_t,j_t}^{(t)})  \otimes \psi_n(v_{k_t,l_t}^{(t)}) \big) \beta_t \right\|_{\lambda,1} 
\\
 & \qquad+ \left\| \displaystyle\sum_{t=k}^{l} \alpha_t \big(\sum_{(k_t,l_t)} \sum_{(i_t,j_t)}  \lambda_{r_t} (\vep^{[r_t]}_{i_t,j_t} , \vep^{[r_t]}_{k_t,l_t})\otimes\phi_n(u_{i_t,j_t}^{(t)})   \otimes \psi_n(v_{k_t,l_t}^{(t))}\big) \beta_t \right\|_{\lambda,1} \\
& \leq \left\|\displaystyle\sum_{t=k}^{l}\displaystyle\alpha_t \big(\sum_{(k_t,l_t)} \sum_{(i_t,j_t)}  \lambda_{r_t} (\vep^{[r_t]}_{i_t,j_t} , \vep^{[r_t]}_{k_t,l_t})\otimes  u_{i_t,j_t}^{(t)} \otimes v_{k_t,l_t}^{(t)}\big) \beta_t \right\|_{\lambda,1} \; \;\;(\text{Using}\;\;\ref{b3})\\
&\qquad \longrightarrow 0 \;\; \text{as}\;\;n \rightarrow \infty,\;\;\text{being the partial sum} \;\; \text{of} \;\;w.\notag
\end{split}
\end{align}

Therefore, one can define an element
\[\begin{array}{c}
  z=   \displaystyle\sum_{t=1}^{\infty} \displaystyle\alpha_t \big(\sum_{(k_t,l_t)} \sum_{(i_t,j_t)}  \lambda_{r_t} (\vep^{[r_t]}_{i_t,j_t} , \vep^{[r_t]}_{k_t,l_t})\otimes  x_{i_t,j_t}^{(t)} \otimes y_{k_t,l_t}^{(t)}\big) \beta_t\in Z(M)\otimes ^{\lambda}Z(N).
\end{array}\]

For sufficiently large choice of $n$ and $\epsilon >0$, we deduce easily that
\begin{align}\label{b4}
\| \phi_{n}\otimes \psi_{n}(w)-z\|_{\lambda}&< \epsilon
\end{align}
Since, $L$  is left invariant by $\phi_{n}\otimes \psi_{n}$ for each $n$, so
\begin{align}
z&=\displaystyle\lim_{n\rightarrow \infty}( \phi_{n}\otimes \psi_{n}(w)\in L \cap (Z(M)\otimes ^{\lambda}Z(N)) \notag
\end{align}

It is easy to show that
$i\circ (\phi_{n}\otimes^{\lambda} \psi_{n})=(\phi_{n}\otimes^{\min} \psi_{n})\circ i$.

\begin{align}\label{a21}
 \|(\phi_{n}\otimes^{\min} \psi_{n})(i(w))-1\otimes1\|_{\min}&=\|(\phi_{n}\otimes^{\min} \psi_{n})(i(w))-(\phi_{n}\otimes^{\min} \psi_{n})(i(1\otimes1))\|_{\min}\notag \\
 &\leq \|i(w)-1\otimes1\|_{\min}< \dfrac{1}{2}.
\end{align}
By (\ref{b4}), we have
\begin{align}\label{a23}
 \|i\circ (\phi_{n}\otimes^{\lambda} \psi_{n})(w)-i(z)\|_{\min}< \epsilon, \text{ for sufficiently large n}.
\end{align}
Then the inequality
\begin{align}\label{a1009}
 \|i(z)-1\otimes 1\|_{\min}\leq \dfrac{1}{2}< 1.
\end{align}
is a consequence of \eqref{a21}, \eqref{a23} and triangle inequality, and so $i(z)$ is invertible in
$L_{\min}\cap (Z(M)\otimes ^{\min}Z(N))$.
Using similar arguments as in (\cite[Theorem 2.11.6, Lemma 2.11.1]{kaniuth} and the fact that $Z(M)$ is a nuclear $C^*$-algebra \cite[ Proposition 1]{geom} we get $Z(M)\otimes^{\lambda}Z(N)$  is semisimple.  The regularity of  $Z(M)\otimes^{\lambda}Z(N)$ follows from \cite[Lemma 4.2.19]{kaniuth}. Since $i(z)$ is invertible in
$L_{\min}\cap (Z(M)\otimes ^{\min}Z(N))$, i.e. invertible in both $L_{\min}$ and $Z(M)\otimes ^{\min}Z(N)$, so
$0\notin \sigma_{Z(M)\otimes ^{\min}Z(N)}(i(z)) $. So  by \cite[Exercise 4.8.12]{kaniuth} $0\notin \sigma_{Z(M)\otimes^{\lambda}Z(N)}(z) $. Hence,
$z$ is invertible in $Z(M)\otimes^{\lambda}Z(N)$, so there exists $w\in Z(M)\otimes ^{\lambda}Z(N)$ such that $zw=wz=1\otimes 1$. Since, $z\in L$ and $L$ being an
ideal, so $1\otimes 1\in L$. \end{proof}

\begin{theorem} Let $\lambda=(\lambda_n)_{n\in \N}$ fulfills \emph{(W2)},
$A$ and $B$ be $C^*$-algebras and let $J$ be a closed ideal in
$A\otimes^{\lambda} B$. If $a \ot  b \in J_{\min}$  then $a \ot b \in J$. In
particular $A\ot^{\lambda} B$ is a *-semi-simple Banach algebra provided $\lambda=(\lambda_n)_{n\in \N}$  further fulfills \emph{(O1)}.
\end{theorem}
\begin{proof}
Suppose that $a, b \geq 0$ and  $a \ot b \in J_{\min}$ but not in $J$.
So by Hahn Banach theorem there exists $\phi\in (A\ot^{\lambda} B)^*$ such that $\phi(J) = 0$ and $\phi(a \ot b) \neq 0$. Since,
$(A\ot^{\lambda} B)^* = CB_{\lambda}(A\times B, \C)$ so $\phi(x \ot y) = \tilde{\phi}(x, y)$ for some $\tilde{\phi} \in CB_{\lambda}(A\times B, \C )$
and for all $x \in A, y \in B$. By Proposition \ref{sc119}, we have $\tilde{\phi}^{**}: A^{**}\times B^{**}\to \C$
a $\lambda$-completely bounded operator satisfying $\|\tilde{\phi}^{**}\|_{ \lambda} = \|\tilde{\phi} \|_{\lambda} $ . 
Let $L$ be the closed ideal in $A^{**} \ot^{\lambda} B^{**}$ generated by $J$. Let $u=\ds\sum_{j=1}^{\infty} \alpha_j \ot_{\lambda_{k_j}}(x_1^j,x_2^j) \beta_j$  be a norm convergent sum in $A\ot^{\lambda} B$ representing a fixed but an arbitrary
element of $J$. Since $\phi$ annihilates $J$, so $\ds\sum_{j=1}^{\infty} \alpha_j \phi_{\lambda_{k_j}}(\ot_{\lambda_{k_j}}(x_1^j,x_2^j)) \beta_j= \ds\sum_{j=1}^{\infty}\ds\sum_{i_1,i_2} \alpha_j \lambda_{k_j}(a_{i_1},a_{i_2}) \ot \phi(x_{i_1}^j \ot x_{i_2}^j) \beta_j=0$. Let
$u, v \in A$ and $s, t \in B$ then  we have $\ds\sum_{j=1}^{\infty}\ds\sum_{i_1,i_2} \alpha_j \lambda_{k_j}(a_{i_1},a_{i_2}) \ot \phi(ux_{i_1}^j v\ot s x_{i_2}^j t) \beta_j=0$. Let $M = A*^{**}$ and $N = B^{**}$ be the von Neumann algebras generated by $A$
and $B$. For each $n\in \N$ and $u\in M$, define $w_n(u)= \ds\sum_{j=1}^{n}\ds\sum_{i_1,i_2} \alpha_j \lambda_{k_{j}}(a_{i_1},a_{i_2}) \ot \phi(ux_{i_1}^j v\ot s x_{i_2}^j t) \beta_j $. We will claim that $\{w_n\}$ is a Cauchy sequence.  To see this, let $m<n$,
$|w_n(u)-w_m(u)|=|\ds\sum_{j=m+1}^{n}\ds\sum_{i_1,i_2} \alpha_j \lambda_{k_{j}}(a_{i_1},a_{i_2}) \ot \phi(ux_{i_1}^j v\ot s x_{i_2}^j t) \beta_j \ | \leq \|u\| \|s\| \|t\|\|v\| \|\ds\sum_{j=m+1}^{n} \alpha_j \ot_{\lambda_{k_{j}}}(x_1^j,x_2^j) \beta_j\|_{\lambda}$, and so $\{w_n\}$ is a Cauchy sequence with limit $w\in M_{*}$ given by 
$w(u)=\ds\sum_{j=1}^{\infty}\ds\sum_{i_1,i_2} \alpha_j \lambda_{k_{j}}(a_{i_1},a_{i_2}) \ot \phi(ux_{i_1}^j v\ot s x_{i_2}^j  t)$.  Again as in \cite{allenideal}, we obtain $\phi$ annihilates $L$.

Now for $\epsilon >0$, let    $p_{\epsilon} \in M$ and $q_{\epsilon} \in N$ be the spectral projections of $a$ and $b$ respectively for the closed interval 
$[\epsilon, \infty )$. Since there is a conditional expectation  from $M$ onto $p_{\epsilon } M p_{\epsilon}$, so $p_{\epsilon } M p_{\epsilon}\ot^{\lambda} q_{\epsilon } N q_{\epsilon}$ is a closed subalgebra of $M\ot^{\lambda} N$ by Lemma \ref{sy6}. Let $L_0= L \cap (p_{\epsilon } M p_{\epsilon}\ot^{\lambda} q_{\epsilon } N q_{\epsilon}) $, a closed ideal in  $p_{\epsilon } M p_{\epsilon}\ot^{\lambda} q_{\epsilon } N q_{\epsilon} $, and so $(L_0)_{\min}$ is a closed ideal in $p_{\epsilon } M p_{\epsilon}\ot^{\min} q_{\epsilon } N q_{\epsilon} $. Now as in \cite[Theorem 4.4]{allenideal} and \cite[Theorem 6]{kumarideal}, we get $(L_0)_{\min}$ contains $p_{\epsilon } \ot q_{\epsilon} $ and so by the above \Cref{lemma}, $p_{\epsilon } \ot q_{\epsilon} \in L_0 $. Hence $L_0= p_{\epsilon } M p_{\epsilon}\ot^{\lambda} q_{\epsilon } N q_{\epsilon}$, which further implies that $p_\epsilon a\ot q_{\epsilon}b \in L$, and so 
$\phi(p_{\epsilon} a \ot  q_{\epsilon}b) = 0$. Letting $\epsilon \to 0$, we have $\phi( a \ot  b) = 0$, 
contrary to the choice of  $\phi$.

In the case when both $a$ and $b$ are arbitrary elements, then one may apply the similar technique  as given in \cite{allenideal} to
obtain the result.\end{proof}


\subsection*{Acknowledgements}
{\small The authors would like to thank Andreas Defant for providing a copy of \cite{wiesnerpolynomials}. The authors would also like to thank the referee for useful comments and suggestions.}

\bibliographystyle{plain}
\bibliography{Matrix_ordering_ref}
\end{document}